\documentclass{amsart}

\usepackage{packages}
\usepackage{commands}
\bibliography{references.bib}

\title{\texorpdfstring{On parahoric $(\calG, \mu)$-displays}{On parahoric (G, mu)-displays}}
\author{Manuel Hoff}
\address{Fakultät für Mathematik, Universität Bielefeld, 33501 Bielefeld, Germany}
\email{manuel.hoff@uni-bielefeld.de}

\begin{document}

    \begin{abstract}
        We develop tools to study spaces of $p$-divisible groups and Abelian varieties with additional structure.
        More precisely, we extend the definition of parahoric (Dieudonné) $(\calG, \bmmu)$-displays given by Pappas to not necessarily $p$-torsionfree base rings and also introduce the notion of an $(m, n)$-truncated $(\calG, \bmmu)$-display.
        Then we study the deformation theory of Dieudonné $(\calG, \bmmu)$-displays.

        As an application we realize the EKOR stratification of the special fiber of a Kisin-Pappas integral Shimura variety of Hodge type as the fibers of a smooth morphism into the algebraic stack of $(2, \onerdt)$-truncated $(\calG, \bmmu)$-displays.
    \end{abstract}

    \maketitle

    \tableofcontents

    \addtocontents{toc}{\protect\setcounter{tocdepth}{1}}
    \section*{Introduction}

This paper is motivated by the study of Shimura varieties, in particular the geometry of their special fibers at places of (parahoric) bad reduction.
These special fibers naturally support certain stratifications and it is an interesting problem of active research to study the properties of these stratifications.

Let us now explain what the so called EKOR-stratification is in the case of the Siegel Shimura variety.
Fix a rational prime $p$, an integer $g \geq 1$ and a subset $J \subseteq \ZZ$ with $J + 2g \ZZ \subseteq J$ and $-J \subseteq J$.
Denote by $\bfK_p \subseteq \GSp_{2g}(\Qp)$ the parahoric subgroup that is the stabilizer of some self-dual $\Zp$-lattice chain of type $J$ in $\Qp^{2g}$.
Also fix a (small enough) compact open subgroup $\bfK^p \subseteq \GSp_{2g}(\adeles^p)$ and set $\bfK \coloneqq \bfK_p \bfK^p$.
The associated Siegel Shimura variety $\Sh_{\bfK} =\Sh_{\bfK}(\GSp_{2g}, S^{\pm})$ then has a moduli description that was first given by de Jong in \cite{de-jong} and then in full generality by Rapoport and Zink in \cite{rapoport-zink}; it parametrizes certain polarized chains of $g$-dimensional Abelian varieties of type $J$ with $\bfK^p$-level structure.
This moduli description gives rise to a natural integral model $\scrS_{\bfK}$ of $\Sh_{\bfK}$ over $\Zp$ (whose special fiber is typically not smooth).

There exists a natural map, called the central leaves map,
\[
    \Upsilon \colon \scrS_{\bfK} \roundbr[\big]{\Fpbar} \to \breve{\bfK}_{p, \sigma} \backslash \GSp_{2g}(\Qpbrev);
\]
here $\breve{\bfK}_p = \calG(\Zpbrev)$ is the group of $\Zpbrev$-valued points of the parahoric group scheme $\calG$ over $\Zp$ corresponding to $\bfK_p$ and $\breve{\bfK}_{p, \sigma}$ denotes the twisted conjugation action $g.x = g x \sigma^{-1}(g)^{-1}$.
It is roughly given by sending a polarized chain of Abelian varieties to the twisted conjugacy class corresponding to the Frobenius $\Phi^{\cov}$ of the associated covariant rational Dieudonné module (more precisely to $p \cdot \sigma^{-1}(\Phi^{\cov, -1})$), see the work of Oort (\cite{oort-foliation}) and He and Rapoport (\cite{he-rapoport}).
The image of this map is given by $\breve{\bfK}_{p, \sigma} \backslash X$ where
\[
    X = \breve{\bfK}_p \Adm_{\calG}(\bmmu_g) \breve{\bfK}_p \subseteq \GSp_{2g}(\Qpbrev)
\]
is a finite union of double cosets.

The fibers of the composition
\[
    \lambda \colon \scrS_{\bfK} \roundbr[\big]{\Fpbar} \xrightarrow{\Upsilon} \breve{\bfK}_{p, \sigma} \backslash X \to \breve{\bfK}_p \backslash X / \breve{\bfK}_p = \Adm_{\calG}(\bmmu_g)
\]
define a stratification of $\scrS_{\bfK, \Fpbar}$ by smooth locally closed subschemes that is called the Kottwitz-Rapoport (KR) stratification.
In fact Rapoport and Zink construct the following data.
\begin{itemize}
    \item
    A flat projective scheme $\bbM^{\loc} = \bbM^{\loc}_{\calG, \bmmu_g}$ over $\Zp$ with a $\calG$-action called the local model, that parametrizes isotropic chains of $g$-dimensional subspaces of the given self-dual lattice chain.
    It satisfies $\bbM^{\loc}(\Fpbar) = X/\breve{\bfK}_p$.

    \item
    A smooth morphism of algebraic stacks
    \[
        \scrS_{\bfK} \to \squarebr[\big]{\calG \backslash \bbM^{\loc}}
    \]
    that parametrizes the Hodge filtration in the de Rham cohomology of a polarized chain of Abelian varieties and gives back the map $\scrS_{\bfK}(\Fpbar) \to \Adm_{\calG}(\bmmu_g)$ after taking $\Fpbar$-valued points.
\end{itemize}

He and Rapoport also consider the composition
\[
    \upsilon \colon \scrS_{\bfK} \roundbr[\big]{\Fpbar} \xrightarrow{\Upsilon} \breve{\bfK}_{p, \sigma} \backslash X \to \breve{\bfK}_{p, \sigma} \backslash \roundbr[\big]{X / \breve{\bfK}_{p, 1}} = \breve{\bfK}_{p, \sigma} (\breve{\bfK}_{p, 1} \times \breve{\bfK}_{p, 1}) \backslash X
\]
where $\breve{\bfK}_{p, 1} \subseteq \breve{\bfK}_p$ is the pro-unipotent radical, and call the (finitely many) fibers of $\upsilon$ Ekedahl-Kottwitz-Oort-Rapoport (EKOR) strata.
These EKOR strata can be shown to be smooth by comparing them to the KR strata at Iwahori level $J = \ZZ$, using a result of Görtz and Hoeve (see \cite{goertz-hoeve}).

In the hyperspecial case $J = 2g\ZZ$ the EKOR stratification is also just called the Ekedahl-Oort (EO) stratification and was first considered by Oort (see \cite{oort-stratification}).
Two points in $\scrS_{\bfK}(\Fpbar)$ corresponding to two polarized Abelian varieties $A$ and $A'$ lie in the same EO stratum if and only if $A[p] \cong A'[p]$.
Viehmann and Wedhorn (see \cite{viehmann-wedhorn}) realize the EO stratification as the fibers of a smooth morphism from $\scrS_{\bfK, \Fp}$ into a certain algebraic stack of zips with symplectic structure (in the sense of Moonen-Wedhorn and Pink-Wedhorn-Ziegler, see \cite{moonen-wedhorn} and \cite{pink-wedhorn-ziegler-zips}).

Thus it is natural to ask the following question.

\begin{question} \label{ques:ekor-morphism}
    Is it possible to naturally realize the map $\upsilon$ (or maybe even $\Upsilon$) as a smooth morphism from $\scrS_{\bfK, \Fp}$ to some algebraic stack that is defined in terms of linear algebraic/group theoretic data?
\end{question}

The existence of such a smooth morphism would in particular give a new proof of the smoothness of the EKOR strata and the closure relations between them.
More importantly it could also provide a tool for studying the geometry of $\scrS_{\bfK, \Fp}$ and cycles on it.

The goal of this article is to give an affirmative answer to Question \ref{ques:ekor-morphism} and its natural generalization to Shimura varieties of Hodge type.
To this end we develop a theory of certain truncated $(\calG, \bmmu)$-displays, extending work of Pappas (\cite{pappas}).
Then, using results of Hamacher and Kim (\cite{hamacher-kim}) and Pappas, we construct a morphism from the $p$-completion of the Shimura variety to the moduli space of these objects and show that it is smooth.

\vspace*{5mm}

Let us give an overview of results that have been obtained so far.

\begin{itemize}
    \item
    Moonen and Wedhorn (\cite{moonen-wedhorn}) introduce the notion of an $F$-zip that is a characteristic $p$ analog of the notion of a Hodge structure.
    Given an Abelian variety $A$ over some $\Fp$-algebra $R$ the de Rham cohomology $H^1_{\dR}(A/R)$ naturally is equipped the structure of an $F$-zip.
    If $R$ is perfect then the datum of $H^1_{\dR}(A/R)$ with its $F$-zip structure is equivalent to the datum of the Dieudonné module of the $p$-torsion $A[p]$.

    Pink, Wedhorn and Ziegler (\cite{pink-wedhorn-ziegler-zips}) define a group theoretic version of the notion of an $F$-zip (that is called $\calG$-zip).

    \item
    Viehmann and Wedhorn (\cite{viehmann-wedhorn}) define a moduli space of $F$-zips with polarization and endomorphism structure (that they call $\calD$-zips) in a PEL-type situation with hyperspecial level structure (that in particular includes the hyperspecial Siegel case).
    They construct a morphism from the special fiber of the associated Shimura variety to this stack of $\calD$-zips that parametrizes the EO stratification.
    Then they show that this morphism is flat and use this to deduce that the EO strata are non-empty and quasi-affine and to compute their dimension (smoothness of the EO strata was already shown by Vasiu in \cite{vasiu-crystalline-boundedness}).

    Zhang (\cite{zhang}) constructs a morphism from the special fiber of the Kisin integral model (see \cite{kisin}) of a Hodge type Shimura variety at hyperspecial level to the group-theoretic stack of $\calG$-zips.
    They then show that this morphism is smooth and thus gives an EO stratification with the desired properties.
    Shen and Zhang (\cite{shen-zhang}) later generalize this to Shimura varieties of Abelian type.

    Hesse (\cite{hesse}) considers an explicit moduli space of chains of $F$-zips with polarization and constructs a morphism from the Siegel modular variety at parahoric level into this stack.
    However it appears that such a morphism is not well-behaved in a non-hyperspecial situation.

    \item
    Xiao and Zhu (\cite{xiao-zhu}) construct a perfect stack of local shtukas $\Sht^{\loc}_{\calG, \bmmu}$ as well as truncated versions $\Sht^{\loc, (m, n)}_{\calG, \bmmu}$ in a hyperspecial situation.
    For a Shimura variety of Hodge type (still at hyperspecial level) they construct a morphism from the perfection of its special fiber into $\Sht^{\loc}_{\calG, \bmmu}$ that gives a central leaves map $\Upsilon$.
    They claim that the induced morphisms to $\Sht^{\loc, (m, n)}_{\calG, \bmmu}$ are perfectly smooth (however, since the diagram on \cite[page 113]{xiao-zhu} is not commutative, the argument in their proof does not seem to work).
    They also construct a natural perfectly smooth forgetful morphism from $\Sht^{\loc, (2, 1)}_{\calG, \bmmu}$ to the perfection of the stack of $\calG$-zips and in particular recover the smoothness result of Zhang after perfection.

    \item
    Shen, Yu and Zhang (\cite{shen-yu-zhang}) generalize the previous work of Zhang, Shen-Zhang and Xiao-Zhu to Shimura varieties of Abelian type at parahoric level (where the construction of integral models is due to Kisin and Pappas, see \cite{kisin-pappas}).
    They give two constructions.
    \begin{itemize}
        \item
        They construct a morphism from each KR stratum into a certain stack of $G$-zips, parametrizing the EKOR strata contained in this KR stratum.
        They also show that this morphism is smooth, thus establishing the smoothness of the EKOR strata.

        \item
        They also construct perfect stacks of local shtukas $\Sht^{\loc}_{\calG, \bmmu}$ and truncated versions $\Sht^{\loc, (m, n)}_{\calG, \bmmu}$ in the parahoric situation and a morphism from the perfection of the special fiber of the Shimura variety to $\Sht^{\loc}_{\calG, \bmmu}$, realizing $\Upsilon$, such that the induced morphisms to $\Sht^{\loc, (m, n)}_{\calG, \bmmu}$ are perfectly smooth (the proof is by the same argument as in \cite{xiao-zhu} and there seems to be the same problem).
    \end{itemize}

    \item
    Zink (\cite{zink}) defines a stack of (not necessarily nilpotent) displays over $\Spf(\Zp)$.
    Displays are semi-linear algebraic objects that are closely related to $p$-divisible groups.
    Bültel and Pappas (\cite{bueltel-pappas}) define group theoretic versions of displays in hyperspecial situations; their objects are called $(\calG, \mu)$-displays.
    In characteristic $p$ this gives a deperfection of the notion of local shtuka from \cite{xiao-zhu}.

    In the parahoric Hodge type situation there is a definition of $(\calG, \bmmu)$-display due to Pappas (see \cite{pappas}) but this definition is only over $p$-torsionfree $p$-adic rings.
    Using results of Hamacher and Kim (\cite{hamacher-kim}) Pappas also constructs a parahoric display on the $p$-completion of the integral model and shows that it is locally universal.

    \item
    In \cite{hoff-siegel} the author defines the notion of a homogeneously polarized chain of displays and certain truncated variants.
    Then they construct a smooth morphism from the special fiber of a Siegel modular variety at parahoric level to the moduli stack of these truncated objects.
    This solves the problem given in Question \ref{ques:ekor-morphism} in the Siegel case.
\end{itemize}

\subsection*{Content}\hfill\\
Let $p$ be a rational prime not equal to $2$.

\subsubsection*{$(m, n)$-truncated displays}\hfill\\
Let $R$ be a $p$-complete ring. 
Denote by $W(R)$ the associated ring of Witt vectors, by $I_R \subseteq W(R)$ the augmentation ideal, i.e.\ the kernel of the projection $W(R) \to R$, by $\sigma \colon W(R) \to W(R)$ the Frobenius morphism and by $\sigma^{\divd} \colon I_R^{\sigma} \to W(R)$ the divided Frobenius, i.e.\ the linearization of the inverse of the Verschiebung morphism.
Having set up this notation, we now recall Zink's definition of a (not necessarily nilpotent) display.

\begin{definition}[{\cite[Definition 1]{zink}}]
    A \emph{display over $R$} is a tuple $(M, M_1, \Phi, \Phi_1)$ that is given as follows.
    \begin{enumerate}
        \item
        $M$ is a finite projective $W(R)$-module.

        \item
        $M_1 \subseteq M$ is a $W(R)$-submodule such that we have $I_R M \subseteq M$ and such that $M_1/I_R M \subseteq M/I_R M$ is a direct summand.

        \item
        $\Phi \colon M^{\sigma} \to M$ and $\Phi_1 \colon M_1^{\sigma} \to M$ are morphisms of $W(R)$-modules such that $\Phi_1$ is surjective and the diagram
        \[
        \begin{tikzcd}[column sep = large]
            (I_R M)^{\sigma} \ar[r] \ar[d]
            & M_1^{\sigma} \ar[d, "\Phi_1"]
            \\
            I_R^{\sigma} \otimes_{W(R)} M^{\sigma} \ar[r, "\sigma^{\divd} \otimes \Phi"]
            & M
        \end{tikzcd}
        \]
        is commutative.
    \end{enumerate}
\end{definition}

There is a natural way of associating to a tuple $(M, M_1)$ satisfying (1) and (2) above (we call such a tuple a \emph{pair}) a finite projective $W(R)$-module $\widetilde{M_1}$ that is given by $\widetilde{M_1} = \im(M_1^{\sigma} \to M^{\sigma})$ in the case that $W(R)$ is $p$-torsionfree.
Giving $(\Phi, \Phi_1)$ is then equivalent to giving an isomorphism of $W(R)$-modules $\Psi \colon \widetilde{M_1} \to M$.
So we see that displays can also be viewed as tuples $(M, M_1, \Psi)$.
This is the point of view that we will use.

When $R$ is a complete Noetherian local ring with perfect residue field of characteristic $p$ then Zink also defines the notion of a \emph{Dieudonné display over $R$} in \cite{zink-dieudonne}.
The definition of a Dieudonné display is the same as that of a display, up to replacing the ring of Witt vectors $W(R)$ with the Zink ring $\What(R) \subseteq W(R)$.

\vspace*{5mm}

Now also fix a tuple $(m, n)$ of positive integers such that $m \geq n + 1$.
Denote by $W_m(R)$ and $W_n(R)$ the rings of truncated Witt vectors of length $m$ and $n$ respectively and by $I_{m, R} \subseteq W_m(R)$ and $I_{n, R} \subseteq W_n(R)$ their augmentation ideals.
The Frobenius $\sigma$ on $W(R)$ then induces $\sigma \colon W_m(R) \to W_n(R)$.
We can now make the following definition.

\begin{definition}[{\Cref{def:display}}]
    An \emph{$(m, n)$-truncated display over $R$} is a tuple $(M, M_1, \Psi)$ that is given as follows.
    \begin{itemize}
        \item
        $M$ is a finite projective $W_m(R)$-module.

        \item
        $M_1 \subseteq M$ is a $W_m(R)$-submodule such that we have $I_{m, R} M \subseteq M_1$ and such that the $R$-submodule $M_1/I_{m, R} M \subseteq M/I_{m, R} M$ is a direct summand.

        \item
        $\Psi \colon \widetilde{M_1} \to W_n(R) \otimes_{W_m(R)} M$ is an isomorphism of $W_n(R)$-modules; here $\widetilde{M_1}$ is a finite projective $W_n(R)$-module that is naturally attached to $(M, M_1)$ similarly as in the non-truncated situation.
    \end{itemize}
\end{definition}

There already exists the notion of an \emph{$n$-truncated display} that is due to Lau in characteristic $p$ (\cite{lau-truncated}) and due to Lau and Zink in general (\cite{lau-zink-truncated}).
The two notions can be related; there exists a natural forgetful functor
\[
    \curlybr[\big]{\text{$(m, n)$-truncated displays}} \to \curlybr[\big]{\text{$n$-truncated displays}}
\]
that can be thought of as forgetting some of the information that is contained in the $W_m(R)$-module $M$ (but it remembers more than just the base change $W_n(R) \otimes_{W_m(R)} M$).
Our $(m, n)$-truncated displays can be thought of as more naive versions of $n$-truncated displays.
But it seems to us that for defining truncated $(\calG, \bmmu)$-displays in a parahoric situation, these $(m, n)$-truncated objects are better suited.

\subsubsection*{$(\calG, \bmmu)$-displays}\hfill\\
Let $(\calG, \bmmu)$ be a local model datum, i.e.\ a tuple consisting of a parahoric $\Zp$-group scheme $\calG$ with generic fiber $G = \calG_{\Qp}$ and a minuscule $G(\Qpbar)$-conjugacy class $\bmmu$ of cocharacters $\mu \colon \bfG_{m, \Qpbar} \to G_{\Qpbar}$.
Denote by $E$ the reflex field of $(\calG, \bmmu)$ and by $\Fq$ its residue field.
By work of Rapoport-Zink, Scholze-Weinstein, Pappas-Zhu, Anschütz-Gleason-Lourenço-Richarz, Fakhruddin-Haines-Lourenço-Richarz and many others we then have a natural associated local model $\bbM^{\loc} = \bbM^{\loc}_{\calG, \bmmu}$ that is a flat projective $\calOE$-scheme with a $\calG$-action whose generic fiber identifies with the homogeneous space $X_{\bmmu}(G)$ of parabolic subgroups of $G_E$ of type $\bmmu^{-1}$.
We denote by $\rmM^{\loc}$ the $p$-completion of $\bbM^{\loc}$.
Let $R$ be a $p$-complete $\calOE$-algebra.

The following definition is given implicitly in Pappas' article (\cite{pappas}).

\begin{definition}[{\Cref{def:g-mu-pairs}}]
    A \emph{$(\calG, \bmmu)$-pair over $R$} is a tuple $(\calP, q)$ consisting of a $\calG$-torsor $\calP$ over $W(R)$ and a $\calG$-equivariant morphism $q \colon \calP_R \to \rmM^{\loc}$.
\end{definition}

Note that $(\calG, \bmmu)$-pairs $(\calP, q)$ are group theoretic versions of pairs $(M, M_1)$ as considered above; the $\calG$-torsor $\calP$ corresponds to the module $M$ and the morphism $q$ corresponds to the filtration $M_1$.
In the case $\calG = \GL_{h, \Zp}$ and $\bmmu = \bmmu'_d = (0^{(h - d)}, (-1)^{(d)})$ we have an equivalence
\[
    \set[\Bigg]{\text{pairs $(M, M_1)$}}{\begin{gathered} \text{$M$ is of rank $h$,} \\ \text{$M_1/I_R M$ is of rank $d$} \end{gathered}} \to \curlybr[\Big]{\text{$(\GL_{h, \Zp}, \bmmu'_d)$-displays}}.
\]

When $(\calG, \bmmu)$ is of Hodge type (see Notation \ref{not:g-mu-displays-group-setup} for more details) and $R$ is $p$-torsionfree then Pappas constructs a natural functor
\[
    \curlybr[\big]{\text{$(\calG, \bmmu)$-pairs}} \to \curlybr[\big]{\text{$\calG$-torsors over $W(R)$}}, \qquad (\calP, q) \mapsto (\calP, q)^{\sim}
\]
that generalizes the construction $(M, M_1) \mapsto \widetilde{M_1}$ to the group-theoretic situation, and consequently is able to define the notion of a $(\calG, \bmmu)$-display over $R$.

\begin{definition}[{\cite[Definition 4.2.2]{pappas}, see also \Cref{def:g-mu-disps}}]
    A \emph{$(\calG, \bmmu)$-display over $R$} is a tuple $(\calP, q, \Psi)$ where $(\calP, q)$ is a $(\calG, \bmmu)$-pair over $R$ and $\Psi \colon (\calP, q)^{\sim} \to \calP$ is an isomorphism of $\calG$-torsors over $W(R)$.
\end{definition}

Similarly, when $R$ is a $p$-torsionfree complete Noetherian local $\calOE$-algebra with perfect residue field of characteristic $p$, then Pappas also defines the notion of a Dieudonné $(\calG, \bmmu)$-display.

\vspace*{5mm}

We observe that Pappas' construction extends to the situation when $R$ is not necessarily $p$-torsionfree, allowing us to also extend the definition of $(\calG, \bmmu)$-display.

We also define the notion of an \emph{$(m, n)$-truncated $(\calG, \bmmu)$-display} for $(m, n)$ as before.
The groupoids of $(m, n)$-truncated $(\calG, \bmmu)$-displays over varying $p$-complete $\calOE$-algebras form a $p$-adic formal algebraic stack
\[ 
    \Disp_{\calG, \bmmu}^{(m, n)}
\]
of finite presentation over $\Spf(\calOE)$.

When $R$ is an $\Fq$-algebra and $m \geq 2$ then we also define the notion of an \emph{$(m, \onerdt)$-truncated $(\calG, \bmmu)$-display over $R$};
here the symbol $\onerdt$ indicates that the isomorphism $\Psi$ is between $\calG_{\Fp}^{\rdt}$-torsors over $R$ (where $\calG_{\Fp}^{\rdt}$ denotes the reductive quotient of the special fiber of $\calG$) instead of between $\calG$-torsors over $R$ as it would be for $(m, 1)$-truncated $(\calG, \bmmu)$-displays.
Again, the groupoids of $(m, \onerdt)$-truncated $(\calG, \bmmu)$-displays over varying $\Fq$-algebras form an algebraic stack $\Disp_{\calG, \bmmu}^{(m, \onerdt)}$ of finite presentation over $\Fq$, and we have a natural bijection
\[
    \abs[\big]{\Disp_{\calG, \bmmu}^{(m, \onerdt)}(\Fpbar)} \to \breve{\bfK}_{p, \sigma} (\breve{\bfK}_{p, 1} \times \breve{\bfK}_{p, 1}) \big\backslash \breve{\bfK}_p \Adm_{\calG}(\bmmu) \breve{\bfK}_p
\]
where $\abs{\blank}$ denotes taking the set of isomorphism classes of a groupoid and $\breve{\bfK}_p = \calG(\Zpbrev)$ similarly as before.

When $\calG$ is a parahoric group scheme for either a general linear or a general symplectic group then we give an explicit description of $(\calG, \bmmu)$-displays in terms of (homogeneously polarized) chains of displays, see also our preprint \cite{hoff-siegel}.
Here we use the description of $\calG$-torsors given by Rapoport and Zink in \cite{rapoport-zink}.

We then study the deformation theory of (Dieudonné) $(\calG, \bmmu)$-displays.
Here our main result is the following.

\begin{theorem}[{\Cref{thm:universal-deformation-g-mu-display}}]
    Let $\calP_0 = (\calP_0, q_0, \Psi_0)$ be a $(\calG, \bmmu)$-display over $\Fpbar$.
    Then $\calP_0$ admits a universal deformation with an explicit description.
\end{theorem}

This theorem can be viewed as a more conceptual version of \cite[Proposition 3.2.17]{kisin-pappas} and some of the arguments we use in the proof are inspired by the arguments of Kisin and Pappas.

We should remark that our result is not unconditional.
We make a certain assumption (see Assumption \ref{hyp:grothendieck-messing}) in order to construct the universal deformation.
In fact we believe that this assumption is always satisfied, and we are able to prove this in a PEL-type situation.
Note that the same assumption is also assumed implicitly in \cite[Subsection 3.2.12]{kisin-pappas}.

\subsubsection*{Application to Shimura varieties}\hfill\\
Let $(\bfG, \bfX)$ be a Shimura datum of Hodge type, let $\bfK = \bfK_p \bfK^p \subseteq \bfG(\adeles)$ be a small enough compact open subgroup such that $\bfK_p$ is a parahoric stabilizer and denote by $\scrS_{\bfK}$ the associated integral Shimura variety over the ring of integers $\calOE$ of the local reflex field $E$ as defined by Kisin and Pappas in \cite{kisin-pappas}.
There is a local model datum $(\calG, \bmmu)$ attached to the Shimura datum $(\bfG, \bfX)$ and the parahoric subgroup $\bfK_p \subseteq \bfG(\Qp)$ and by work of Hamacher and Kim (\cite{hamacher-kim}) and Pappas the $p$-completion $\widehat{\scrS_{\bfK}}$ naturally supports a $(\calG, \bmmu)$-display.

Using results of from \cite{kisin-pappas} and \cite{pappas} we show the following theorem.

\begin{theorem}[{\Cref{thm:main-result-hodge-type}}]
    Let $(m, n)$ be a tuple of positive integers with $m \geq n + 1$.
    Then the morphism
    \[
        \widehat{\scrS_{\bfK}} \to \Disp_{\calG, \bmmu}^{(m, n)}
    \]
    is smooth.
    Similarly, for a positive integer $m \geq 2$ also the morphism
    \[
        \scrS_{\bfK, \Fq} \to \Disp_{\calG, \bmmu}^{(m, \onerdt)}
    \]
    is smooth.
\end{theorem}

The geometric fibers of the morphism $\scrS_{\bfK, \Fq} \to \Disp_{\calG, \bmmu}^{(m, \onerdt)}$ are precisely the EKOR strata, so that this in particular answers Question \ref{ques:ekor-morphism} affirmatively.

After passing to to the perfection of the special fiber our theorem recovers the result \cite[Theorem 4.4.3]{shen-yu-zhang} by Shen, Yu and Zhang (whose proof seems problematic to us as was already remarked earlier), see \Cref{cor:main-result-perfection}, at least up to a slightly different normalization in the definition of $(m, n)$-restricted local shtukas.

Similarly, after restricting to a single KR stratum, we also recover \cite[Theorem 3.4.11]{shen-yu-zhang}, see \Cref{cor:main-result-kr-local}.

\subsection*{Acknowledgements}\hfill\\
I am grateful to my supervisor Ulrich Görtz for his constant support throughout my PhD studies.
I would also like to thank Sebastian Bartling, Jochen Heinloth, Pol van Hoften, Giulio Marazza, George Pappas and Torsten Wedhorn for helpful discussions and suggestions.

This work was partially funded by the DFG Graduiertenkolleg 2553.
Up to minor changes it agrees with a part of the authors PhD thesis \cite{hoff-thesis}.

\subsection*{Notation}

\begin{itemize}
    \item
    The rings in this text are always assumed to be commutative and unital.

    \item
    The symbols $\QQ$, $\RR$ and $\CC$ denote the fields of rational, real and complex numbers.
    We write $\Qbar$ for the algebraic closure of $\QQ$ inside $\CC$.

    \item
    We fix a rational prime $p$ not equal to $2$ and denote by $\Qp$, $\Zp$ and $\Fp$ the field of $p$-adic numbers, its ring of integers and its residue field.
    We fix an algebraic closure $\Qpbar$ of $\Qp$ with residue field $\Fpbar$ and an embedding $\Qbar \to \Qpbar$.

    \item
    Let $R$ be a $\Zp$-algebra.
    We write $W(R)$ for the associated ring of Witt vectors (that naturally is again a $\Zp$-algebra) and
    \[
        I_R \coloneqq \ker \roundbr[\big]{W(R) \to R} \subseteq W(R)
    \]
    for its augmentation ideal.
    We also write $\sigma \colon W(R) \to W(R)$ for the Frobenius and $\sigma^{\divd} \colon I_R^{\sigma} \to W(R)$ for the linearization of the inverse of the Verschiebung (that we also call the divided Frobenius because it satisfies $p \sigma^{\divd}(1 \otimes x) = \sigma(x)$ for $x \in I_R$).

    Let $n$ be a positive integer.
    We write $W_n(R)$ for the ring of $n$-truncated Witt vectors and $I_{n, R} \subseteq W_n(R)$ for its augmentation ideal.
    When $m$ is a second positive integer with $m \geq n + 1$ we again have a Frobenius $\sigma \colon W_m(R) \to W_n(R)$ and a divided Frobenius
    \[
        \sigma^{\divd} \colon W_n(R) \otimes_{\sigma, W_m(R)} I_{m, R} \to W_n(R).
    \]

    \item
    Let $R$ be a complete Noetherian local ring with perfect residue field $k$ (that is always assumed to be of characteristic $p$).
    We write $\What(R) \subseteq W(R)$ for the associated Zink ring (see \cite{zink-dieudonne}), $\Ihat_R \subseteq \What(R)$ for its augmentation ideal and $\sigma \colon \What(R) \to \What(R)$ and $\sigma^{\divd} \colon (\Ihat_R)^{\sigma} \to \What(R)$ for its Frobenius and divided Frobenius similarly as before.

    \item
    Let $\calG$ be a parahoric group scheme over $\Zp$ with generic fiber $G = \calG_{\Qp}$.
    We denote by $L^+ \calG$ the (Witt vector) positive loop group of $\calG$, i.e.\ the affine group scheme over $\Zp$ given by $(L^+ \calG)(R) = \calG(W(R))$.
    For a positive integer $n$ we also write $L^{(n)} \calG$ for the $n$-truncated positive loop group of $\calG$, i.e.\ the smooth affine group scheme over $\Zp$ given by $(L^{(n)} \calG)(R) = \calG(W_n(R))$.
    Finally we write $L^{(\onerdt)} \calG = \calG_{\Fp}^{\rdt}$ for the reductive quotient of the special fiber $\calG_{\Fp}$.

    We write $\bbL^+ \calG$ and $\bbL^{(n)} \calG$ for the perfection of the special fibers of $L^+ \calG$ and $L^{(n)} \calG$ respectively and we write $\bbL G$ for the loop groop, i.e.\ for the sheaf of groups on the category of perfect $\Fp$-algebras (that we equip with the flat topology) given by $(\bbL G)(R) \coloneqq G(W(R)[1/p])$.

    \item
    Let $\Lambda$ be a finite free $\Zp$-module of rank $h$ and let $\calG \subseteq \GL(\Lambda)$ be a smooth closed $\Zp$-subgroup scheme.
    Then by a $\calG$-structure on a finite projective $R$-module of rank $h$ (for some $\Zp$-algebra $R$) we mean a choice of reduction of the associated $\GL(\Lambda)$-torsor to a $\calG$-torsor.
\end{itemize}
    \addtocontents{toc}{\protect\setcounter{tocdepth}{2}}
    \section{Displays} \label{sec:displays}

Recall that classical Dieudonné theory gives a natural equivalence of categories
\[
    \curlybr[\big]{\text{$p$-divisible groups over $R$}}^{\op} \to \curlybr[\big]{\text{Dieudonné modules over $R$}}
\]
for any perfect ring $R$ of characteristic $p$.
Here a Dieudonné module is a tuple $(M, M_1, \Psi)$ consisting of
\begin{itemize}
    \item
    a finite projective $W(R)$-module $M$,
    
    \item
    a $W(R)$-submodule $M_1 \subseteq M$ with $p M \subseteq M_1 \subseteq M$ and such that $M_1/pM \subseteq M/pM$ is a direct summand, and
    
    \item
    an isomorphism of $W(R)$-modules $\Psi \colon M_1^{\sigma} \to M$.
\end{itemize}

The notion of a (not necessarily nilpotent) display, introduced by Zink in \cite{zink} and studied extensively by Lau and Zink, gives a generalization of the theory of Dieudonné modules to arbitrary $p$-complete base rings.
Just as in classical Dieudonné theory, there is a natural way to associate a display to a $p$-divisible group.
However, this construction typically does not yield an equivalence of categories anymore (although it does if one restricts to formal $p$-divisible groups).

If one considers only complete Noetherian local rings with perfect residue field of characteristic $p$ as base rings there exists a theory of Dieudonné displays, that refines the theory of (usual) displays.
There is again a natural way to associate a Dieudonné display to a $p$-divisible groups and Zink showed in \cite{zink-dieudonne} that this in fact does give an equivalence of categories.

In this chapter we recall the theory of displays and Dieudonné displays.
However we phrase it in a slightly non-standard way, similarly as in \cite{kisin-pappas}, so that the definition of a display looks almost like the definition of a Dieudonné module given above.
The equivalence of our definitions and the ones from \cite{zink} and \cite{zink-dieudonne} follows from the argument in \cite[Lemma 3.1.5]{kisin-pappas}.
We also introduce the new related notion of an $(m, n)$-truncated display (for positive integers $m$ and $n$ with $m \geq n + 1$) that is inspired by the definition of $(m, n)$-restricted local shtukas given by Xiao and Zhu in \cite{xiao-zhu}.
Our $(m, n)$-truncated displays are in some sense more naive versions of the $n$-truncated displays introduced by Lau and Zink in \cite{lau-truncated} and \cite{lau-zink-truncated}, but they are better suited for our purposes.

\begin{notation}
    We fix integers $0 \leq d \leq h$.
    The letter $R$ always denotes a $p$-complete ring, $m$ always denotes a positive integer and $(m, n)$ always denotes a tuple of positive integers with $m \geq n + 1$.
\end{notation}

\subsection{Pairs and displays}

\begin{definition} \label{def:pair}
    A \emph{pair (of type $(h, d)$) over $R$} is a tuple $(M, M_1)$ consisting of a finite projective $W(R)$-module $M$ (of rank $h$) and a $W(R)$-submodule $M_1 \subseteq M$ with $I_R M \subseteq M_1$ and such that $M_1/I_R M \subseteq M/I_R M$ is a direct summand (of rank $d$).

    An \emph{$m$-truncated pair over $R$} is a tuple $(M, M_1)$ consisting of a finite projective $W_m(R)$-module $M$ and a $W_m(R)$-submodule $M_1 \subseteq M$ with $I_{m, R} M \subseteq M_1$ and such that $M_1/I_{m, R} M \subseteq M/I_{m, R} M$ is a direct summand.

    Now let $R$ be a complete Noetherian local ring with perfect residue field (of characteristic $p$).
    A \emph{Dieudonné pair over $R$} is a tuple $(M, M_1)$ consisting of a finite projective $\What(R)$-module $M$ and a $\What(R)$-submodule $M_1 \subseteq M$ with $\Ihat_R M \subseteq M_1$ and such that $M_1/\Ihat_R M \subseteq M/\Ihat_R M$ is a direct summand.
\end{definition}

\begin{definition}
    Let $(M, M_1)$ and $(M', M'_1)$ be two pairs over $R$.
    Then a morphism $(M, M_1) \to (M', M'_1)$ is a morphism of $W(R)$-modules $f \colon M \to M'$ such that $f(M_1) \subseteq M'_1$.
    In the same way we also define morphisms of $m$-truncated pairs and Dieudonné pairs.
\end{definition}

\begin{remark} \label{rmk:base-changing-pairs}
    Let $(M, M_1)$ be a pair over $R$ and let $R \to R'$ be a morphism of $p$-complete rings.
    Then we can form the base change $(M', M'_1) = (M, M_1)_{R'}$ that is a pair over $R'$.
    It is characterized by
    \[
        M' = W(R') \otimes_{W(R)} M \quad \text{and} \quad M'_1/I_{R'} M' = R' \otimes_R (M_1/I_R M).
    \]
    Similarly we can also base change $m$-truncated pairs and Dieudonné pairs.
\end{remark}

\begin{remark}
    There are natural truncation functors
    \[
        \curlybr[\big]{\text{pairs over $R$}} \to \curlybr[\big]{\text{$m$-truncated pairs over $R$}}
    \]
    and
    \[
        \curlybr[\big]{\text{$m'$-truncated pairs over $R$}} \to \curlybr[\big]{\text{$m$-truncated pairs over $R$}}
    \]
    for $m \leq m'$.
    When $R$ is a complete Noetherian local ring with perfect residue field, then there is also a natural functor
    \[
        \curlybr[\big]{\text{Dieudonné pairs over $R$}} \to \curlybr[\big]{\text{pairs over $R$}}.
    \]
\end{remark}

\begin{remark} \label{rmk:normal-decomposition}
    Let $(M, M_1)$ be a pair over $R$.
    Then $(M, M_1)$ always has a \emph{normal decomposition $(L, T)$}, i.e.\ a direct sum decomposition $M = L \oplus T$ such that $M_1 = L \oplus I_R T$.
    Given a second pair $(M', M'_1)$ over $R$ with normal decomposition $(L', T')$, every morphism of pairs $f \colon (M, M_1) \to (M', M'_1)$ can be written in matrix form $f = \begin{psmallmatrix} a & b \\ c & d \end{psmallmatrix}$ with
    \[
        a \colon L \to L', \quad b \colon T \to L', \quad c \colon L \to I_R T', \quad d \colon T \to T'.
    \]
    The same is true for $m$-truncated pairs and Dieudonné pairs.
\end{remark}

At this point we would like to naively define a display over $R$ as a tuple $(M, M_1, \Psi)$ where $(M, M_1)$ is a pair over $R$ and $\Psi \colon M_1^{\sigma} \to M$ is an isomorphism of $W(R)$-modules.
However, the $W(R)$-module $M_1^{\sigma}$ is typically not finite projective, so this definition does not behave well.

The situation can be remedied by introducing a suitable replacement $\widetilde{M_1}$ of $M_1^{\sigma}$.

\begin{proposition} \label{prop:m-1-tilde}
    There are unique functors
    \[
    \begin{array}{r c l}
        \curlybr[\Big]{\text{pairs (of type $(h, d)$) over $R$}}
        & \to
        & \curlybr[\Bigg]{\begin{gathered}\text{finite projective $W(R)$-modules} \\ \text{(of rank $h$)}\end{gathered}}, \\[5mm]
        (M, M_1)
        & \mapsto
        & \widetilde{M_1}
    \end{array}
    \]
    together with natural isomorphisms $\widetilde{M_1}[1/p] \to M^{\sigma}[1/p]$ that are compatible with base change in $R$ and that in the case that $W(R)$ is $p$-torsionfree are given by
    \[
        \widetilde{M_1} = \im \roundbr[\big]{M_1^{\sigma} \to M^{\sigma}}.
    \]

    There also exist unique functors
    \[
    \begin{array}{r c l}
        \curlybr[\big]{\text{$m$-truncated pairs over $R$}}
        & \to
        & \curlybr[\big]{\text{finite projective $W_n(R)$-modules}}, \\[1mm]
        (M, M_1)
        & \mapsto
        & \widetilde{M_1}
    \end{array}
    \]
    that are compatible with base change in $R$ and with passing from pairs to $m$-truncated pairs (where we recall that $m$ and $n$ are positive integers with $m \geq n + 1$).

    When we restrict to complete Noetherian local rings $R$ with perfect residue field there also exist unique functors
    \[
    \begin{array}{r c l}
        \curlybr[\big]{\text{Dieudonné pairs over $R$}}
        & \to
        & \curlybr[\big]{\text{finite projective $\What(R)$-modules}}, \\[1mm]
        (M, M_1)
        & \mapsto
        & \widetilde{M_1}
    \end{array}
    \]
    together with natural isomorphisms $\widetilde{M_1}[1/p] \to M^{\sigma}[1/p]$ that are compatible with base change in $R$ and that in the case that $\What(R)$ is $p$-torsionfree are again given by the formula $\widetilde{M_1} = \im(M_1^{\sigma} \to M^{\sigma})$.
    This construction is compatible with passing from Dieudonné pairs to pairs.
\end{proposition}

\begin{proof}
    Let $(M, M_1)$ be a pair over $R$ and choose a normal decomposition $(L, T)$, see \Cref{rmk:normal-decomposition}.
    We then define
    \[
        \widetilde{M_1} \coloneqq L^{\sigma} \oplus T^{\sigma}.
    \]
    Given a second pair $(M', M'_1)$ with normal decomposition $(L', T')$, a morphism of pairs
    \[
        f \colon (M, M_1) \to (M', M'_1)
    \]
    can be written in matrix form $f = \begin{psmallmatrix} a & b \\ c & d \end{psmallmatrix}$ with
    \[
        a \colon L \to L', \quad b \colon T \to L', \quad c \colon L \to I_R T', \quad d \colon T \to T'.
    \]
    We then define $\widetilde{f} \colon \widetilde{M_1} \to \widetilde{M'_1}$ by the matrix $\begin{psmallmatrix} a^{\sigma} & p b^{\sigma} \\ \dot{c} & d^{\sigma} \end{psmallmatrix}$ where $\dot{c}$ denotes the composition
    \[
        \dot{c} \colon L^{\sigma} \xrightarrow{c^{\sigma}} I_R^{\sigma} \otimes_{W(R)} T'^{\sigma} \xrightarrow{\sigma^{\divd} \otimes \id} T'^{\sigma}.
    \]
    Finally we define the natural isomorphism
    \[
        \widetilde{M_1} \squarebr[\big]{1/p} = L^{\sigma}[1/p] \oplus T^{\sigma}[1/p] \to M^{\sigma}[1/p] = L^{\sigma}[1/p] \oplus T^{\sigma}[1/p], \qquad l + t \mapsto l + p t.
    \]
    One can now check that this is well-defined and has the required properties.

    For $m$-truncated pairs the construction of $(M, M_1) \mapsto \widetilde{M_1}$ is essentially the same as for pairs; here
    \[
        \dot{c} \colon W_n(R) \otimes_{\sigma, W_m(R)} L \to W_n(R) \otimes_{\sigma, W_m(R)} T'
    \]
    is given as the composition
    \begin{gather*}
        W_n(R) \otimes_{\sigma, W_m(R)} L \xrightarrow{c^{\sigma}} \roundbr[\big]{W_n(R) \otimes_{\sigma, W_m(R)} I_{m, R}} \otimes_{W_n(R)} \roundbr[\big]{W_n(R) \otimes_{\sigma, W_m(R)} T'} \\
        \xrightarrow{\sigma^{\divd} \otimes \id} W_n(R) \otimes_{\sigma, W_m(R)} T'.
    \end{gather*}

    Finally, for Dieudonné pairs the construction of $(M, M_1) \mapsto \widetilde{M_1}$ is exactly the same as for pairs.
\end{proof}

\begin{remark}\label{rmk:m-1-tilde-m-1-sigma}
    Let $(M, M_1)$ be a pair over $R$.
    Then there is a natural surjective morphism of $W(R)$-modules $M_1^{\sigma} \to \widetilde{M_1}$.
    Typically this morphism is not an isomorphism but it is when $R$ is a perfect $\Fp$-algebra.
\end{remark}

\begin{remark} \label{rmk:pairs-inclusions}
    For a pair $(M, M_1)$ over $R$ the isomorphism $\widetilde{M_1}[1/p] \to M^{\sigma}[1/p]$ can be refined to natural morphisms
    \[
        M^{\sigma} \to \widetilde{M_1} \quad \text{and} \quad \widetilde{M_1} \to M^{\sigma}
    \]
    coming from the inclusions $p M^{\sigma} \subseteq \widetilde{M_1} \subseteq M^{\sigma}$ in the case that $W(R)$ is $p$-torsionfree.
    The same remark also applies to Dieudonné pairs.
\end{remark}

\begin{definition} \label{def:display}
    A \emph{display over $R$} is a tuple $(M, M_1, \Psi)$ where $(M, M_1)$ is a pair over $R$ and $\Psi \colon \widetilde{M_1} \to M$ is an isomorphism of $W(R)$-modules.

    Similarly, an \emph{$(m, n)$-truncated display over $R$} is a tuple $(M, M_1, \Psi)$ where $(M, M_1)$ is an $m$-truncated pair over $R$ and $\Psi \colon \widetilde{M_1} \to W_n(R) \otimes_{W_m(R)} M$ is an isomorphism of $W_n(R)$-modules.

    Now let $R$ be a complete Noetherian local ring with perfect residue field.
    A \emph{Dieudonné display over $R$} is a tuple $(M, M_1, \Psi)$ where $(M, M_1)$ is a Dieudonné pair over $R$ and $\Psi \colon \widetilde{M_1} \to M$ is an isomorphism of $\What(R)$-modules.
\end{definition}

\begin{definition} \label{def:display-frobenius}
    Let $(M, M_1, \Psi)$ be a display over $R$.
    Then the \emph{Frobenius of $(M, M_1, \Psi)$} is the composition
    \[
        \Phi \colon M^{\sigma} \to \widetilde{M_1} \xrightarrow{\Psi} M,
    \]
    where $M^{\sigma} \to \widetilde{M_1}$ is the morphism from \Cref{rmk:pairs-inclusions}.
    Note that $\Phi$ becomes an isomorphism after inverting $p$.

    We make the same definition for Dieudonné displays.
\end{definition}

\subsection{Duals and twists} \label{sec:duals-twists}

\begin{definition} \label{def:dual-pairs}
    Let $(M, M_1)$ be a pair over $R$.
    Then we define its \emph{dual}
    \[
        (M, M_1)^{\vee} \coloneqq \roundbr[\big]{M^{\vee}, M_1^{\ast}}
    \]
    as follows.
    \begin{itemize}
        \item
        $M^{\vee} = \Hom_{W(R)}(M, W(R))$ is the dual of the finite projective $W(R)$-module $M$.

        \item
        $M_1^{\ast} \subseteq M$ is the $W(R)$-submodule of all $\omega \colon M \to W(R)$ such that $\omega(M_1) \subseteq I_R$.
        Equivalently it is the preimage under
        \[
            M^{\vee} \to M^{\vee}/I_R M^{\vee} \cong (M/I_R M)^{\vee}
        \]
        of the orthogonal complement $(M_1/I_R M)^{\perp} \subseteq (M/I_R M)^{\vee}$.
    \end{itemize}
    This endows the category of pairs over $R$ with a duality.
    We similarly define duals of $m$-truncated pairs and Dieudonné pairs.
\end{definition}

\begin{remark}
    Note that if $(M, M_1)$ is a pair of type $(h, d)$ over $R$ then $(M, M_1)^{\vee}$ is of type $(h, h - d)$.
\end{remark}

\begin{lemma} \label{lem:m-1-tilde-duality}
    Given a pair $(M, M_1)$ over $R$ there exists a unique natural isomorphism
    \[
        \widetilde{M_1^{\ast}} \to \widetilde{M_1}^{\vee}
    \]
    that is compatible with base change in $R$ and makes the diagram
    \[
    \begin{tikzcd}
        \widetilde{M_1^{\ast}} \squarebr[\big]{1/p} \ar[r] \ar[d]
        & \widetilde{M_1}^{\vee} \squarebr[\big]{1/p} \ar[d]
        \\
        M^{\sigma, \vee}[1/p] \ar[r, "p^{-1}"]
        & M^{\sigma, \vee}[1/p]
    \end{tikzcd}
    \]
    commutative.
    The same is true for Dieudonné pairs.

    Given an $m$-truncated pair $(M, M_1)$ over $R$ there exists a unique natural isomorphism
    \[
        \widetilde{M_1^{\ast}} \to \widetilde{M_1}^{\vee}
    \]
    that is compatible with base change in $R$ and with passing from pairs to $m$-truncated pairs.

    In other words, the various functors $(M, M_1) \mapsto \widetilde{M_1}$ from \Cref{prop:m-1-tilde} are compatible with dualities.
\end{lemma}

\begin{proof}
    Let $(L, T)$ be a normal decomposition of $(M, M_1)$.
    Then $(T^{\vee}, L^{\vee})$ is a normal decomposition of $(M^{\vee}, M_1^{\ast})$.
    We then define the desired natural isomorphism to be the composition
    \[
        \widetilde{M_1^{\ast}} \cong T^{\vee, \sigma} \oplus L^{\vee, \sigma} \cong \roundbr[\big]{L^{\sigma} \oplus T^{\sigma}}^{\vee} \cong \widetilde{M_1}^{\vee}. \qedhere
    \]
\end{proof}

\begin{definition} \label{def:dual-displays}
    Let $(M, M_1, \Psi)$ be a display over $R$.
    Then we define its \emph{dual}
    \[
        (M, M_1, \Psi)^{\vee} \coloneqq \roundbr[\big]{M^{\vee}, M_1^{\ast}, \Psi^{\vee, -1}},
    \]
    where $\Psi^{\vee, -1}$ really denotes the composition
    \[
        \widetilde{M_1^{\ast}} \to \widetilde{M_1}^{\vee} \xrightarrow{\Psi^{\vee, -1}} M^{\vee},
    \]
    the first isomorphism being the one from \Cref{lem:m-1-tilde-duality}.
    This endows the category of pairs over $R$ with a duality.
    We similarly define duals of $(m, n)$-truncated displays and Dieudonné displays.
\end{definition}

\begin{definition} \label{def:twist-pairs}
    Let $(M, M_1)$ be a pair over $R$ and let $I$ be an invertible $W(R)$-module.
    Then we define the \emph{twist}
    \[
        I \otimes (M, M_1) \coloneqq (I \otimes_{W(R)} M, I \otimes_{W(R)} M_1).
    \]
    This endows the category of pairs over $R$ with an action of the symmetric monoidal category of invertible $W(R)$-modules and this action is compatible with the duality from \Cref{def:dual-pairs}.
    We similarly define twists of $m$-truncated displays and Dieudonné displays.
\end{definition}

\begin{lemma} \label{lem:m-1-tilde-action}
    Given a pair $(M, M_1)$ over $R$ and an invertible $W(R)$-module $I$ there exists a unique natural isomorphism
    \[
        \roundbr[\big]{I \otimes_{W(R)} M_1}^{\sim} \cong I^{\sigma} \otimes_{W(R)} \widetilde{M_1}
    \]
    that is compatible with base change in $R$ and makes the diagram
    \[
    \begin{tikzcd}
        (I \otimes_{W(R)} M_1)^{\sim} \squarebr[\big]{1/p} \ar[r] \ar[d]
        & \roundbr[\big]{I^{\sigma} \otimes_{W(R)} \widetilde{M_1}} \squarebr[\big]{1/p} \ar[d]
        \\
        (I \otimes_{W(R)} M)^{\sigma}[1/p] \ar[r]
        & I^{\sigma}[1/p] \otimes_{W(R)[1/p]} M^{\sigma}[1/p]
    \end{tikzcd}
    \]
    commutative.
    The same is true for Dieudonné pairs.

    Given an $m$-truncated pair $(M, M_1)$ over $R$ and an invertible $W_m(R)$-module $I$ there exists a unique natural isomorphism
    \[
        (I \otimes_{W_m(R)} M_1)^{\sim} \cong (W_n(R) \otimes_{\sigma, W_m(R)} I) \otimes_{W_n(R)} \widetilde{M_1}.
    \]
    that is compatible with base change in $R$ and with passing from pairs to $m$-truncated pairs.

    In other words, the various functors $(M, M_1) \mapsto \widetilde{M_1}$ from \Cref{prop:m-1-tilde} are equivariant with respect to the functors of symmetric monoidal categories
    \[
    \begin{array}{r c l}
        \curlybr[\big]{\text{invertible $W(R)$-modules}}
        & \to
        & \curlybr[\big]{\text{invertible $W(R)$-modules}}, \\[1 mm]
        I
        & \mapsto
        & I^{\sigma},
    \end{array}
    \]
    \[
    \begin{array}{r c l}
        \curlybr[\big]{\text{invertible $W_m(R)$-modules}}
        & \to
        & \curlybr[\big]{\text{invertible $W_n(R)$-modules}}, \\[1 mm]
        I
        & \mapsto
        & W_n(R) \otimes_{\sigma, W_m(R)} I,
    \end{array}
    \]
    and, when $R$ is a complete Noetherian local ring with residue field $\Fpbar$,
    \[
    \begin{array}{r c l}
        \curlybr[\big]{\text{invertible $\What(R)$-modules}}
        & \to
        & \curlybr[\big]{\text{invertible $\What(R)$-modules}}, \\[1 mm]
        I
        & \mapsto
        & I^{\sigma}.
    \end{array}
    \]
\end{lemma}

\begin{proof}
    Let $(L, T)$ be a normal decomposition of $(M, M_1)$.
    Then $(I \otimes_{W(R)} L, I \otimes_{W(R)} T)$ is a normal decomposition of $I \otimes (M, M_1)$.
    We then define the desired natural isomorphism to be the composition
    \[
        (L \otimes_{W(R)} M_1)^{\sim} \cong (I \otimes_{W(R)} L)^{\sigma} \oplus (I \otimes_{W(R)} T)^{\sigma} \cong I^{\sigma} \otimes_{W(R)} (L^{\sigma} \oplus T^{\sigma}) \cong I^{\sigma} \otimes_{W(R)} \widetilde{M_1}. \qedhere
    \]
\end{proof}

\begin{definition} \label{def:twist-displays}
    Let $(M, M_1, \Psi)$ be a display over $R$ and let $(I, \iota)$ be a tuple consisting of an invertible $W(R)$-module $I$ and an isomorphism $\iota \colon I^{\sigma} \to I$.
    Then we define the twist
    \[
        (I, \iota) \otimes (M, M_1, \Psi) \coloneqq (I \otimes_{W(R)} M, I \otimes_{W(R)} M_1, \iota \otimes \Psi),
    \]
    where $\iota \otimes \Psi$ really denotes the composition
    \[
        (I \otimes_{W(R)} M_1)^{\sim} \to I^{\sigma} \otimes_{W(R)} \widetilde{M_1} \xrightarrow{\iota \otimes \Psi} M,
    \]
    the first isomorphism being the one from \Cref{lem:m-1-tilde-action}.
    This endows the category of displays over $R$ with an action of the symmetric monoidal category of tuples $(I, \iota)$ as above and this action is compatible with the duality from \Cref{def:dual-displays}.

    We similarly define twists of $(m, n)$-truncated displays and Dieudonné displays.
    For $(m, n)$-truncated displays over $R$ the twists are by tuples $(I, \iota)$ consisting of an invertible $W_m(R)$-module $I$ and an isomorphism $\iota \colon W_n(R) \otimes_{\sigma, W_m(R)} I \to I$.
\end{definition}

\subsection{Grothendieck-Messing Theory for Dieudonné displays}

\begin{notation} \label{not:grothendieck-messing}
    In this subsection $S \to R$ always denotes a surjection of complete Noetherian local rings with perfect residue field such that its kernel $\fraka \subseteq S$ is equipped with nilpotent divided powers that are compatible with the canonical divided powers on $p S \subseteq S$ and continuous in the sense that there exists a fundamental system of open ideals $\frakb \subseteq S$ such that the divided power structure on $S$ extends to the quotient $S/\frakb$.
\end{notation}

\begin{definition} \label{def:dieudonne-pair-grothendieck-messing}
    A \emph{Dieudonné pair for $S/R$} is a tuple $(M, M_1)$ consisting of a finite free $\What(S)$-module $M$ and a $\What(R)$-submodule $M_1 \subseteq \What(R) \otimes_{\What(S)} M$ such that $(\What(R) \otimes_{\What(S)} M, M_1)$ is a Dieudonné pair over $R$.
\end{definition}

\begin{lemma} \label{lem:m-1-grothendieck-messing}
    The functor
    \[
    \begin{array}{r c l}
        \curlybr[\big]{\text{Dieudonné pairs over $S$}}
        & \to
        & \curlybr[\big]{\text{finite free $\What(S)$-modules}}, \\[1mm]
        (M, M_1)
        & \mapsto
        & \widetilde{M_1}
    \end{array}
    \]
    admits a natural factorization over the category of Dieudonné pairs for $S/R$.
    We denote the induced functor again by
    \[
    \begin{array}{r c l}
        \curlybr[\big]{\text{Dieudonné pairs for $S/R$}}
        & \to
        & \curlybr[\big]{\text{finite free $\What(S)$-modules}}, \\[1mm]
        (M, M_1)
        & \mapsto
        & \widetilde{M_1}.
    \end{array}
    \]
    This construction is compatible with base change in $S/R$.
\end{lemma}

\begin{proof}
    Write $\Ihat_{S/R}$ for the kernel of the projection $\What(S) \to R$.
    As described in \cite[Section 2]{zink-dieudonne} the divided powers give rise to a $\What(S)$-linear inclusion $\fraka \subseteq \What(S)$ and $\Ihat_{S/R}$ decomposes as $\Ihat_{S/R} = \Ihat_S \oplus \fraka$.
    We can therefore define a divided Frobenius
    \[
        \sigma^{\divd} \colon \Ihat_{S/R}^{\sigma} = \Ihat_S^{\sigma} \oplus \fraka^{\sigma} \xrightarrow{(\sigma^{\divd}, 0)} \What(S).
    \]
    As $\fraka$ is contained in the kernel of $\sigma \colon \What(S) \to \What(S)$ we still have $p \sigma^{\divd}(1 \otimes x) = \sigma(x)$ for $x \in \Ihat_{S/R}$.
    At this point we can now proceed as in the proof of \Cref{prop:m-1-tilde} to construct the desired functor.
\end{proof}

\begin{definition} \label{def:dieudonne-display-grothendieck-messing}
    A \emph{Dieudonné display for $S/R$} is a tuple $(M, M_1, \Psi)$ where $(M, M_1)$ is a Dieudonné pair for $S/R$ and $\Psi \colon \widetilde{M_1} \to M$ is an isomorphism of $\What(S)$-modules.
\end{definition}

\begin{theorem}[Zink] \label{thm:grothendieck-messing}
    The natural forgetful functor
    \[
        \curlybr[\big]{\text{Dieudonné displays for $S/R$}} \to \curlybr[\big]{\text{Dieudonné displays over $R$}}
    \]
    is an equivalence of categories.
\end{theorem}

\begin{proof}
    This is a reformulation of \cite[Theorem 3]{zink-dieudonne}.
\end{proof}

\begin{remark}
    It is again possible to define duals and twists of Dieudonné pairs for $S/R$ and Dieudonné displays for $S/R$, similarly as in Subsection \ref{sec:duals-twists}.
\end{remark}

\subsection{The universal deformation of a display}\hfill\\ \label{sec:universal-deformation-display}
The goal of this subsection is to give a construction of a universal deformation of a display over a perfect field $k$ to a Dieudonné display, see \Cref{thm:universal-deformation-display}.
Here we roughly follow \cite[Subsection 3.1]{kisin-pappas}.

\begin{notation}
    Let $F/\Qp$ be a complete discretely valued field with perfect residue field $k$ and write $W \coloneqq W(k)$.
    Let $(M_0, M_{0, 1}, \Psi_0)$ be a display of type $(h, d)$ over $k$.

    We consider the deformation problem $\Def$ over $\calO_F$, i.e. the set-valued functor on the category of complete Noetherian local $\calO_F$-algebras with residue field $k$, that is given by
    \[
        \Def \colon R \mapsto \curlybr[\Bigg]{\begin{gathered} \text{deformations $(M, M_1, \Psi)$ of $(M_0, M_{0, 1}, \Psi_0)$ over $R$} \\ \text{as a Dieudonné display} \end{gathered}}.
    \]
\end{notation}

\begin{proposition} \label{prop:deformation-problem-display-pro-representable}
    The deformation problem $\Def$ is pro-representable.
\end{proposition}

\begin{proof}
    We check that Schlessinger's criterion applies (see \cite[Theorem 2.11]{schlessinger}).

    \begin{itemize}
        \item
        Clearly we have $\Def(k) = \curlybr{\ast}$.

        \item
        Suppose that we are given a diagram
        \[
        \begin{tikzcd}
            & R'' \ar[d]
            \\
            R' \ar[r]
            & R
        \end{tikzcd}
        \]
        of Artinian local $\calO_F$-algebras with residue field $k$ with $R' \to R$ surjective and assume without loss of generality that  $R' \to R$ is a square zero extension.
        Also set $R''' \coloneqq R' \times_R R''$ and equip the kernels of the morphisms $R' \to R$ and $R''' \to R''$ with the trivial divided powers.

        We have to show that the map
        \[
            \Def(R''') \to \Def(R') \times_{\Def(R)} \Def(R'')
        \]
        is a bijection.

        So let $(M'', M''_1, \Psi'') \in \Def(R'')$ and write $(M, M_1, \Psi)$ for its base change to $R$.
        By \Cref{thm:grothendieck-messing} $(M'', M''_1, \Psi'')$ lifts uniquely to a Dieudonné display $(M''', M''_1, \Psi''')$ for $R'''/R''$.
        In the same way $(M, M_1, \Psi)$ lifts uniquely to a Dieudonné display $(M', M_1, \Psi')$ for $R'/R$ and in fact $(M', M_1, \Psi')$ is the base change of $(M''', M''_1, \Psi''')$ along the morphism $(R''' \to R'') \to (R' \to R)$.

        Now lifting $(M'', M''_1, \Psi'')$ to an object in $\Def(R''')$ is the same as giving a lift of $M''_1$ to a filtration $M'''_1 \subseteq M'''$ (making $(M''', M'''_1)$ a Dieudonné pair) while lifting $(M, M_1, \Psi)$ to an object in $\Def(R')$ is the same as giving a lift of $M_1$ to a filtration $M'_1 \subseteq M'$.
        Thus the claim follows because $R''' = R' \times_R R''$.

        \item
        \Cref{thm:grothendieck-messing} implies that the tangent space $T_{\ast} \Def$ at the base point $\ast \in \Def(k)$ is naturally isomorphic to the tangent space of the Grassmannian $\Grass_{M_0, d}$ at the point
        \[
            \roundbr[\big]{M_{0, 1}/pM_0 \subseteq M_0/pM_0 \cong M_{0, k}} \in \Grass_{M_0, d}(k).
        \]
        In particular it is finite-dimensional. \qedhere 
    \end{itemize}
\end{proof}

\begin{construction} \label{con:universal-deformation}
    Write $R^{\univ}$ for the completed local ring of $\Grass_{M_0, d, \calO_F}$ at the $k$-point corresponding to $M_{0, 1}/pM_0 \subseteq M_{0, k}$.
    Set
    \[
        M^{\univ} \coloneqq M_{0, \What(R^{\univ})}
    \]
    and let
    \[
        M^{\univ}_1 \subseteq M^{\univ}
    \]
    be the preimage of the universal direct summand of $M^{\univ}/\Ihat_{R^{\univ}} M^{\univ} \cong M_{0, R^{\univ}}$.
    We thus have a Dieudonné pair
    \[
        \roundbr[\big]{M^{\univ}, M^{\univ}_1}
    \]
    over $R^{\univ}$ that is a deformation of $(M_0, M_{0, 1})$.
    Set
    \[
        \fraka \coloneqq \frakm_{R^{\univ}}^2 + \frakm_F R^{\univ} \subseteq R^{\univ}
    \]
    and equip $\frakm_{R^{\univ}}/\fraka \subseteq R^{\univ}/\fraka$ with the trivial divided powers.
    
    Now consider the composition
    \[
        \What(R^{\univ}/\fraka) \otimes_{\What(R^{\univ})} \widetilde{M^{\univ}_1} \to \roundbr[\big]{\widetilde{M_{0, 1}}}_{\What(R^{\univ}/\fraka)} \xrightarrow{\Psi_0} M_{0, \What(R^{\univ}/\fraka)} \to \What(R^{\univ}/\fraka) \otimes_{\What(R^{\univ})} M^{\univ},
    \]
    where the first isomorphism is the one given by \Cref{lem:m-1-grothendieck-messing}.
    We choose a lift of this composition to an isomorphism
    \[
        \Psi^{\univ} \colon \widetilde{M^{\univ}_1} \to M^{\univ}
    \]
    so that we obtain a deformation $\roundbr[\big]{M^{\univ}, M^{\univ}_1, \Psi^{\univ}} \in \Def(R^{\univ})$.
\end{construction}

\begin{theorem} \label{thm:universal-deformation-display}
    The deformation $(M^{\univ}, M^{\univ}_1, \Psi^{\univ}) \in \Def(R^{\univ})$ from Construction~\ref{con:universal-deformation} is universal.
\end{theorem}

\begin{proof}
    By \Cref{prop:deformation-problem-display-pro-representable} we know that $\Def$ is pro-representable.
    From the construction and \Cref{thm:grothendieck-messing} it follows that the morphism $\Spf(R^{\univ}) \to \Def$ induces an isomorphism on tangent spaces at the respective unique $k$-points (see also the third bullet point in the proof of \Cref{prop:deformation-problem-display-pro-representable}).
    As $R^{\univ}$ is a power series ring over $\calO_F$ this is enough to conclude that in fact $\Spf(R^{\univ}) \to \Def$ is an isomorphism.
\end{proof}

\subsection{\texorpdfstring{Displays and $p$-divisible groups}{Displays and p-divisible groups}}

\begin{theorem}[Lau, Zink] \label{thm:p-div-groups-disps}
    There exists a natural functor
    \[
        \curlybr[\big]{\text{$p$-divisible groups over $R$}}^{\op} \to \curlybr[\big]{\text{displays over $R$}}
    \]
    that is compatible with dualities and coincides with (contravariant) Dieudonné theory when $R$ is a perfect $\Fp$-algebra.
    If $X$ is a $p$-divisible group of height $h$ and dimension $d$ then the associated display is of type $(h, d)$.

    Let $R$ be a complete Noetherian local ring with perfect residue field.
    Then there is a natural equivalence
    \[
        \curlybr[\big]{\text{$p$-divisible groups over $R$}}^{\op} \to \curlybr[\big]{\text{Dieudonné displays over $R$}}
    \]
    that is again compatible with dualities.
    The two constructions are compatible in the sense that the diagram
    \[
    \begin{tikzcd}[row sep = tiny]
        & \curlybr[\big]{\text{Dieudonné displays over $R$}} \ar[dd]
        \\
        \curlybr[\big]{\text{$p$-divisible groups over $R$}}^{\op} \ar[ru] \ar[rd] &
        \\
        & \curlybr[\big]{\text{displays over $R$}}
    \end{tikzcd}
    \]
    is commutative.
\end{theorem}

\begin{proof}
    This is proven by Lau and Zink in \cite[Proposition 2.1]{lau-truncated} and \cite{zink-dieudonne}.
\end{proof}

\begin{remark}
    The functors in \Cref{thm:p-div-groups-disps} are not exactly the ones from \cite{lau-truncated} and \cite{zink-dieudonne}, as we choose the contravariant normalization (following \cite{kisin-pappas}) instead of the covariant one.
\end{remark}
    \section{\texorpdfstring{$(\calG, \bmmu)$-displays}{(G, mu)-displays}} \label{sec:g-mu-displays}

In \cite{bueltel-pappas}, Bültel and Pappas introduced the notion of a $(\calG, \mu)$-display for a reductive group scheme $\calG$ over $\Zp$ and a minuscule cocharacter $\mu$ of $\calG$ defined over the ring of integers of some finite unramified extension of $\Qp$.
One recovers the usual notion of a display by setting $\calG = \GL_{h, \Zp}$.
In a Hodge type situation, Pappas generalized this in \cite{pappas} to the case where $\calG$ is a parahoric group scheme, but he only works over $p$-torsionfree $p$-complete base rings.

In this chapter we recall the definitions of (Dieudonné) $(\calG, \bmmu)$-displays from \cite{pappas} and generalize them to general $p$-complete base rings (respectively complete Noetherian local rings with perfect residue field).
We also introduce the new notion of an $(m, n)$-truncated $(\calG, \bmmu)$-display.
Then we study the deformation theory of Dieudonné $(\calG, \bmmu)$-displays and in particular explicitly construct a universal deformation for a $(\calG, \bmmu)$-display over a perfect field $k$.
Finally we give an application to the Kisin-Pappas integral models of Shimura varieties of Hodge type at parahoric level defined in \cite{kisin-pappas} and specifically to the EKOR stratification that was introduced by He and Rapoport in \cite{he-rapoport}.

\begin{notation} \label{not:g-mu-displays-group-setup}
    We fix a local model datum $(\calG, \bmmu)$, i.e.\ a tuple consisting of a parahoric group scheme $\calG$ over $\Zp$ with generic fiber $G = \calG_{\Qp}$ and a minuscule geometric conjugacy class of cocharacters of $G$.
    Write $E \subseteq \Qpbar$ for its reflex field and let $q$ be the cardinality of the residue field $\calOE/\frakm_E = \Fq$.
    Associated to $(\calG, \bmmu)$ we have, by the work of Rapoport-Zink, Pappas-Zhu, Scholze-Weinstein, and recently Anschütz-Gleason-Lourenço-Richarz and Fakhruddin-Haines-Lourenço-Richarz (and many more), the local model $\bbM^{\loc} = \bbM^{\loc}_{\calG, \bmmu}$, that is a flat projective normal $\calOE$-scheme that comes equipped with an action of $\calG$ (see \cite{anschuetz-gleason-lourenco-richarz} and \cite{fakhruddin-haines-lourenco-richarz}).
    We write $\rmM^{\loc}$ for its $p$-completion.

    We also fix a finite free $\Zp$-module $\Lambda$ of rank $h$ and a closed immersion of $\Zp$-algebraic groups $\iota \colon \calG \to \GL(\Lambda)$ such that the following conditions are satisfied.
    \begin{itemize}
        \item
        $\iota(\bmmu) = \bmmu_d$ for some $0 \leq d \leq h$.
        Here $\bmmu_d$ denotes the conjugacy class of cocharacters of $\GL(\Lambda)_{\Qpbar}$ that induce a weight decomposition $\Lambda_{\Qpbar} = \Lambda_0 \oplus \Lambda_{-1}$ with $\Lambda_{-1}$ of dimension $d$.

        \item
        $\bfG_{m, \Zp} \subseteq \iota(\calG)$.

        \item
        The induced closed immersion $X_{\bmmu}(G) \to X_{\bmmu_d}(\GL(\Lambda))_E$ extends to a closed immersion
        \[
            \bbM^{\loc} \to \roundbr[\big]{\bbM^{\loc}_{\GL(\Lambda), \bmmu_d}}_{\calOE} = \Grass_{\Lambda, d, \calOE}.
        \]
        This morphism is then automatically equivariant for $\iota \colon \calG \to \GL(\Lambda)$.
    \end{itemize}
    Whenever it is convenient we omit $\iota$ from our notation and simply view it as an inclusion $\calG \subseteq \GL(\Lambda)$.

    In this chapter $R$ always denotes a $p$-complete $\calO_E$-algebra.
    Moreover, $m$ always denotes a positive integer and $(m, n)$ always denotes a tuple of positive integers with $m \geq n + 1$.
    Sometimes (with explicit indication) we allow $n$ to take the additional value $\onerdt$ (recall that we write $L^{(\onerdt)} \calG$ for the reductive quotient of $\calG_{\Fp}$), in which case we require $m \geq 2$.
\end{notation}

\begin{remark} \label{rmk:existence-iota-condition}
    Note that in Notation~\ref{not:g-mu-displays-group-setup} the existence of $(\Lambda, \iota)$ with the required properties restricts the possible choices for $(\calG, \bmmu)$.
    By \cite[Proposition 3.1.6]{pappas}, the following conditions are sufficient for the existence of $(\Lambda, \iota)$.
    \begin{itemize}
        \item
        There exists a finite-dimensional $\Qp$-vector space $W$ and an injective morphism of $\Qp$-algebraic groups $\kappa \colon G \to \GL(W)$ such that $\kappa(\bmmu) = \bmmu_d$ for some $d$ and $\bfG_{m, \Qp} \subseteq \kappa(G)$.

        \item
        $G$ splits over a tamely ramified extension of $\Qp$.

        \item
        $p$ does not divide the order of $\pi_1(G_{\der, \Qpbar})$.

        \item
        $\calG$ is a parahoric stabilizer.
    \end{itemize}
\end{remark}

\subsection{\texorpdfstring{$(\calG, \bmmu)$-pairs and $(\calG, \bmmu)$-displays}{(G, mu)-pairs and (G, mu)-displays}}

\begin{definition} \label{def:g-mu-pairs}
    A \emph{$(\calG, \bmmu)$-pair over $R$} is a tuple $(\calP, q)$ consisting of a $\calG$-torsor $\calP$ over $W(R)$ and a $\calG$-equivariant morphism $q \colon \calP_R \to \rmM^{\loc}$.

    An \emph{$m$-truncated $(\calG, \bmmu)$-pair over $R$} is a tuple $(\calP, q)$ consisting of a $\calG$-torsor $\calP$ over $W_m(R)$ and a $\calG$-equivariant morphism $q \colon \calP_R \to \rmM^{\loc}$.

    Now let $R$ be a complete Noetherian local $\calOE$-algebra with perfect residue field.
    A \emph{Dieudonné $(\calG, \bmmu)$-pair over $R$} is a tuple $(\calP, q)$ consisting of a $\calG$-torsor $\calP$ over $\What(R)$ and a $\calG$-equivariant morphism $q \colon \calP_R \to \rmM^{\loc}$.
\end{definition}

\begin{remark} \label{rmk:g-mu-pair-pair}
    Giving a $(\calG, \bmmu)$-pair over $R$ is the same as giving a pair $(M, M_1)$ of type $(h, d)$ over $R$ (see \Cref{def:pair}) together with a $\calG$-structure on $M$ such that for every (or equivalently one) local trivialization of $M$ the $R$-point of $(\Grass_{\Lambda, d, \calOE})^{\wedge}_p$ parametrizing $M_1$ is contained in $\rmM^{\loc}$.
    Similar statements also hold for truncated $(\calG, \bmmu)$-pairs and Dieudonné $(\calG, \bmmu)$-pairs.

    In the following we will abuse notation and refer to a pair $(M, M_1)$ that is implicitly equipped with a $\calG$-structure as above as a $(\calG, \bmmu)$-pair.
\end{remark}

\begin{lemma} \label{lem:g-mu-pairs-stack}
    The groupoids of $(\calG, \bmmu)$-pairs over varying $p$-complete $\calO_E$-algebras naturally form a stack $\Pair_{\calG, \bmmu}$ over $\Spf(\calO_E)$ and we have a natural equivalence
    \[
        \Pair_{\calG, \bmmu} \cong \squarebr[\big]{L^+ \calG \backslash \rmM^{\loc}}.
    \]
    Similarly, also $m$-truncated $(\calG, \bmmu)$-pairs form a ($p$-adic formal algebraic) stack $\Pair^{(m)}_{\calG, \bmmu}$ over $\Spf(\calO_E)$ that is naturally equivalent to the quotient stack $[L^{(m)} \calG \backslash \rmM^{\loc}]$.
\end{lemma}

\begin{proof}
    This is immediate from the definition.
\end{proof}


\begin{proposition} \label{prop:g-mu-m-1-tilde}
    There are unique functors
    \[
    \begin{array}{r c l}
        \curlybr[\big]{\text{$(\calG, \bmmu)$-pairs over $R$}}
        & \to
        & \curlybr[\big]{\text{$\calG$-torsors over $W(R)$}}, \\[1mm]
        (M, M_1)
        & \mapsto
        & \widetilde{M_1}
    \end{array}
    \]
    together with natural isomorphisms $\widetilde{M_1}[1/p] \to M^{\sigma}[1/p]$ that are compatible with base change in $R$ and with passing to pairs (without $(\calG, \bmmu)$-structure), see \Cref{rmk:g-mu-pair-pair} and \Cref{prop:m-1-tilde}.

    There exist unique functors
    \[
    \begin{array}{r c l}
        \curlybr[\big]{\text{$m$-truncated $(\calG, \bmmu)$-pairs over $R$}}
        & \to
        & \curlybr[\big]{\text{$\calG$-torsors over $W_n(R)$}}, \\[1mm]
        (M, M_1)
        & \mapsto
        & \widetilde{M_1}
    \end{array}
    \]
    that are compatible with base change and with passing from $(\calG, \bmmu)$-pairs to $m$-truncated $(\calG, \bmmu)$-pairs.
    These functors are also compatible with passing from $m$-truncated $(\calG, \bmmu)$-pairs to $m$-truncated pairs.

    When we restrict to complete Noetherian local $\calOE$-algebras $R$ with perfect residue field there also exist unique functors
    \[
    \begin{array}{r c l}
        \curlybr[\big]{\text{Dieudonné $(\calG, \bmmu)$-pairs over $R$}}
        & \to
        & \curlybr[\big]{\text{$\calG$-torsors over $\What(R)$}}, \\[1mm]
        (M, M_1)
        & \mapsto
        & \widetilde{M_1}
    \end{array}
    \]
    together with natural isomorphisms $\widetilde{M_1}[1/p] \to M^{\sigma}[1/p]$ that are compatible with base change in $R$ and with passing to Dieudonné pairs.
    This construction is compatible with passing from Dieudonné $(\calG, \bmmu)$-pairs to $(\calG, \bmmu)$-pairs.
\end{proposition}

\begin{proof}
    The statement about truncated $(\calG, \bmmu)$-pairs follows immediately from the one about non-truncated $(\calG, \bmmu)$-pairs.
    The remaining part of the claim is essentially \cite[Proposition 4.1.10]{pappas}; for completeness we give a sketch of the argument.

    After passing to the universal case and applying \cite[Corollary 3.2.6 and Remark 3.2.7]{pappas} we are reduced to considering Dieudonné $(\calG, \bmmu)$-pairs in the case $R = \calO_K$ for some finite extension $K/E$.
    We then conclude by the following key lemma that is a reformulation of \cite[Lemma 3.2.9]{kisin-pappas}.
\end{proof}

\begin{lemma} \label{lem:g-mu-m-1-tilde-key-lemma}
    Let $K/E$ be a complete discretely valued extension with perfect residue field and let $(M, M_1)$ be a Dieudonné $(\calG, \bmmu)$-pair over $\calO_K$.
    Then there exists a (unique) $\calG$-structure on $\widetilde{M_1}$ such that the isomorphism $\widetilde{M_1}[1/p] \to M^{\sigma}[1/p]$ is compatible with $\calG$-structures.
\end{lemma}

\begin{proof}
    After possibly enlarging $K$ we may assume that there exists a trivialization $M \cong \Lambda_{\What(\calO_K)}$ and that there exists a representative $\mu \colon \bfG_{m, K} \to G_K$ of $\bmmu$ that gives rise to the point in $\bbM^{\loc}(K) \cong \rmM^{\loc}(\calO_K)$ corresponding to $M_1$.

    At this point we set up some notation.
    Let $k$ be the residue field of $K$ and write $W \coloneqq W(k)$ and $K_0 \coloneqq W[1/p]$.
    Set $\frakS \coloneqq W \doublesquarebr{u}$ and let $\sigma \colon \frakS \to \frakS$ be the Frobenius lift given by
    \[
        \sigma \roundbr[\Big]{\sum_i a_i u^i} \coloneqq \sum_i \sigma(a_i) u^{pi}.
    \]
    Fix a uniformizer $\varpi_K \in K$ and let $P(u) \in \frakS$ be its minimal Eisenstein polynomial over $K_0$.
    Then we have the $\sigma$-equivariant morphism
    \[
        \alpha \colon \frakS \to \What(\calO_K), \qquad \sum_i a_i u^i \mapsto \sum_i a_i [\varpi_K^i].
    \]

    Note that we have $\alpha(P(u)) = P([\varpi_K]) \in \Ihat_{\calO_K}$ because its image in $\calO_K$ is $P(\varpi_K) = 0$.
    Moreover we have
    \[
        \sigma^{\divd} (1 \otimes P([\varpi_K])) \in \What(\calO_K)^{\times}.
    \]
    Indeed its image in $W$ is given by $\sigma^{\divd}(1 \otimes b_0) \in W^{\times}$ where $b_0 \in p W^{\times}$ is the zeroth coefficient of $P(u)$.
    Thus we see that
    \[
        \sigma \roundbr[\big]{P([\varpi_K])} \What(\calO_K) = p \What(\calO_K).
    \]

    Now set $N \coloneqq \Lambda_{\frakS}$ and let $N_1 \subseteq N$ be the preimage of the direct summand of $N/P(u) N \cong \Lambda_{\calO_K}$ corresponding to the given point in $M^{\loc}(\calO_K)$.
    Then $N_1$ is again a free $\frakS$-module of rank $h$ and the above computation implies that the tautological isomorphism $\What(R) \otimes_{\frakS} N^{\sigma} \cong M^{\sigma}$ restricts to an isomorphism $\What(R) \otimes_{\frakS} N_1^{\sigma} \cong \widetilde{M_1}$.
    Thus to prove the claim it suffices to show that there exists a $\calG$-structure on $N_1$ such that the isomorphism $N_1[1/P(u)] \cong N[1/P(u)]$ is compatible with $\calG$-structures.

    Next we observe that the morphism $\frakS \to \frakS/P(u) \frakS \cong \calO_K$ induces a morphism
    \[
        \widehat{\frakS}_0 \coloneqq \frakS[1/p]^{\wedge}_{P(u)} \to K
    \]
    that admits a unique section by Hensel's Lemma, making $\widehat{\frakS}_0$ into a $K$-algebra.
    Thus it makes sense to consider the element $P(u) \mu(P(u)) \in G (\widehat{\frakS}_0[1/P(u)])$.
    By definition of $N_1$ this element induces an isomorphism
    \[
        P(u) \mu(P(u)) \colon \widehat{\frakS}_0 \otimes_{\frakS} N \to \widehat{\frakS}_0 \otimes_{\frakS} N_1
    \]
    that we use to define a $\calG$-structure on $\widehat{\frakS}_0 \otimes_{\frakS} N_1$.
    Over $\widehat{\frakS}_0[1/P(u)]$ this $\calG$-structure agrees with the one coming from the isomorphism $N_1[1/P(u)] \cong N[1/P(u)]$ so that we can apply Beauville-Laszlo gluing (see for example \cite[Tag 05ET]{stacks-project}) to obtain a $\calG$-structure on $N_1[1/p]$ that is compatible with the given one on $N_1[1/P(u)]$.
    By the purity result \cite[Proposition 10.3]{anschuetz} by Anschütz it follows that these $\calG$-structures extend to a $\calG$-structure on $N_1$ as desired.
\end{proof}

\begin{remark} \label{rmk:g-mu-m-1-tilde-independent-iota}
    The functors $(M, M_1) \mapsto \widetilde{M_1}$ from \Cref{prop:g-mu-m-1-tilde} are independent of the choice of embedding $\iota$ (see \cite[Remark 4.1.11]{pappas}).
\end{remark}

\begin{remark} \label{rmk:g-mu-m-1-tilde-pel}
    Let $(M, M_1)$ be a $(\calG, \bmmu)$-pair.
    The definition of the $\calG$-structure on $\widetilde{M_1}$ is rather indirect.
    When $W(R)$ is $p$-torsionfree then it is given by extending the given $\calG$-structure on $\widetilde{M_1}[1/p] \cong M^{\sigma}[1/p]$, but if that is not the case then one first needs to lift $(M, M_1)$ to some $(\calG, \bmmu)$-pair over a $p$-complete $\calOE$-algebra $R' \to R$ such that $W(R')$ is $p$-torsionfree.

    Now suppose that $\calG \subseteq \GL(\Lambda)$ is cut out by a family of endomorphisms and (possibly) a homogeneous polarization (this includes a lot of PEL-cases).
    Then the situation is much nicer in the sense that it is possible to give an explicit description of the $\calG$-structure on $\widetilde{M_1}$, in particular without having to lift to a situation where $W(R)$ is $p$-torsionfree:
    The $\calG$-structure on a $(\calG, \bmmu)$-pair $(M, M_1)$ is encoded by a family of endomorphisms and a homogeneous polarization on that pair $(M, M_1)$.
    By functoriality (and using Subsection \ref{sec:duals-twists}) these then induce a family of endomorphisms and a homogeneous polarization on $\widetilde{M_1}$ that encode the $\calG$-structure from \Cref{prop:g-mu-m-1-tilde}.
    The same is also true for truncated $(\calG, \bmmu)$-pairs and Dieudonné $(\calG, \bmmu)$-pairs.

    In general $\calG \subseteq \GL(\Lambda)$ will be cut out by a family of tensors, each one of which is an endomorphism of $\Lambda^{\otimes \ell}$ for some $\ell$.
    Now the problem with generalizing the approach above seems to be that in general there is no good way to produce an endomorphism of $\widetilde{M_1}^{\otimes \ell}$ from an endomorphism of $M^{\otimes \ell}$ (that maybe preserves some sort of filtration that is induced from $M_1 \subseteq M$).
\end{remark}

\begin{remark} \label{rmk:mloc-infty}
    By \Cref{lem:g-mu-pairs-stack} the datum of the functors
    \[
        \curlybr[\big]{\text{$(\calG, \bmmu)$-pairs over $R$}} \to \curlybr[\big]{\text{$\calG$-torsors over $W(R)$}}
    \]
    from \Cref{prop:g-mu-m-1-tilde} for varying $R$ can be equivalently described as an $L^+ \calG$-equivariant $L^+ \calG$-torsor over $\rmM^{\loc}$ that we denote by $\rmM^{\loc, (\infty)}$.

    We write $\rmM^{\loc, (n)}$ for the reduction of this $L^+ \calG$-torsor to an $L^{(n)} \calG$-torsor.
    The $L^+ \calG$-equivariant structure on $\rmM^{\loc, (n)}$ factors through $L^{(n + 1)} \calG$.

    We also write $\rmM^{\loc, (\onerdt)}$ for the reduction of $\rmM^{\loc, (\infty)}_{\Fq}$ to an $L^{(\onerdt)} \calG$-torsor over $\rmM^{\loc}_{\Fq}$ (where we recall that $L^{(\onerdt)} \calG = \calG_{\Fp}^{\rdt}$ denotes the reductive quotient of the special fiber of $\calG$).
    The $L^+ \calG$-equivariant structure on $\rmM^{\loc, (\onerdt)}$ factors through $L^{(2)} \calG$.
\end{remark}

\begin{definition} \label{def:g-mu-disps}
    A \emph{$(\calG, \bmmu)$-display over $R$} is a tuple $(M, M_1, \Psi)$ where $(M, M_1)$ is a $(\calG, \bmmu)$-pair over $R$ and $\Psi \colon \widetilde{M_1} \to M$ is an isomorphism of $\calG$-torsors over $W(R)$.

    An \emph{$(m, n)$-truncated $(\calG, \bmmu)$-display over $R$} is a tuple $(M, M_1, \Psi)$ where $(M, M_1)$ is an $m$-truncated $(\calG, \bmmu)$-pair over $R$ and $\Psi \colon \widetilde{M_1} \to M_{W_n(R)}$ is an isomorphism of $\calG$-torsors over $W_n(R)$.

    Similarly, when $R$ is an $\Fq$-algebra, an \emph{$(m, \onerdt)$-truncated $(\calG, \bmmu)$-display over $R$} is a tuple $(M, M_1, \Psi)$ where $(M, M_1)$ is an $m$-truncated $(\calG, \bmmu)$-pair over $R$ and $\Psi \colon \calG_{\Fp}^{\rdt} \times^{\calG_{\Fp}} (\widetilde{M_1})_R \to \calG_{\Fp}^{\rdt} \times^{\calG_{\Fp}} M_R$ is an isomorphism of $\calG_{\Fp}^{\rdt}$-torsors over $R$.

    Let $R$ be a complete Noetherian local $\calO_{\Ebrev}$-algebra with residue field $\Fpbar$.
    A \emph{Dieudonné $(\calG, \bmmu)$-display over $R$} is a tuple $(M, M_1, \Psi)$ where $(M, M_1)$ is a Dieudonné $(\calG, \bmmu)$-pair over $R$ and $\Psi \colon \widetilde{M_1} \to M$ is an isomorphism of $\calG$-torsors over $\What(R)$.
\end{definition}

\begin{remark}
    Giving a $(\calG, \bmmu)$-display over $R$ is the same as giving a display $(M, M_1, \Psi)$ of type $(h, d)$ over $R$ (see \Cref{def:display}) together with a $\calG$-structure on $M$ such that $(M, M_1)$ is actually a $(\calG, \bmmu)$-pair and the isomorphism $\Psi \colon \widetilde{M_1} \to M$ is compatible with the $\calG$-structures on both sides.

    In the situation where $W(R)$ is $p$-torsionfree this last condition is equivalent to saying that the Frobenius $\Phi \colon M^{\sigma}[1/p] \to M[1/p]$ (see \Cref{def:display-frobenius}) is compatible with the $\calG$-structures on both sides.

    The same remark also applies to Dieudonné $(\calG, \bmmu)$-displays.
\end{remark}

\begin{lemma} \label{lem:g-mu-disps-stack}
    The groupoids of $(\calG, \bmmu)$-displays over varying $p$-complete $\calO_E$-algebras naturally form a stack $\Disp_{\calG, \bmmu}$ over $\Spf(\calO_E)$ and we have a natural equivalence
    \[
        \Disp_{\calG, \bmmu} \cong \squarebr[\big]{(L^+ \calG)_{\Delta} \backslash \rmM^{\loc, (\infty)}}
    \]
    where the subscript $\Delta$ indicates that we take the quotient by the diagonal action of $L^+ \calG$ on $\rmM^{\loc, (\infty)}$.

    Similarly, also $(m, n)$-truncated $(\calG, \bmmu)$-displays form a $p$-adic formal algebraic stack $\Disp_{\calG, \bmmu}^{(m, n)}$ and we have a natural equivalence
    \[
        \Disp_{\calG, \bmmu}^{(m, n)} \cong [(L^{(m)} \calG)_{\Delta} \backslash \rmM^{\loc, (n)}].
    \]
    This also applies when $n = \onerdt$, in which case $\Disp_{\calG, \bmmu}^{(m, \onerdt)}$ is an algebraic stack over $\Spec(\Fq)$.
\end{lemma}

\begin{remark}
    Suppose that we are given a second local model datum $(\calG', \bmmu')$ and a closed immersion $(\calG, \bmmu) \to (\calG', \bmmu')$.
    Suppose furthermore that we have a closed immersion $\iota' \colon \calG' \to \GL(\Lambda)$ that satisfies the assumptions from Notation \ref{not:g-mu-displays-group-setup} and makes the diagram
    \[
    \begin{tikzcd}
        \calG \ar[rr] \ar[rd, swap, "\iota"]
        && \calG' \ar[ld, "\iota'"]
        \\
        & \GL(\Lambda)
    \end{tikzcd}
    \]
    commutative.

    Then there are natural functors
    \[
        \curlybr[\big]{\text{$(\calG, \bmmu)$-pairs over $R$}} \to \curlybr[\big]{\text{$(\calG', \bmmu')$-pairs over $R$}}
    \]
    and the construction $(M, M_1) \mapsto \widetilde{M_1}$ is compatible with these change-of-group functors.
    Consequently we also obtain natural functors
    \[
        \curlybr[\big]{\text{$(\calG, \bmmu)$-displays over $R$}} \to \curlybr[\big]{\text{$(\calG', \bmmu')$-displays over $R$}}.
    \]
    The same remark also applies to the truncated and Dieudonné versions of displays.
\end{remark}

\begin{lemma} \label{lem:restriction-smooth}
    For $(m, n) \leq (m', n')$ the natural forgetful morphism
    \[
        \Disp_{\calG, \bmmu}^{(m', n')} \to \Disp_{\calG, \bmmu}^{(m, n)}
    \]
    is smooth.
    This is also true when $n$ and/or $n'$ take the value $\onerdt$, in which case the domain of the morphism may need to be base changed to $\Fq$.
\end{lemma}

\begin{proof}
    This follows directly from the smoothness of the projection $\rmM^{\loc, (n')} \to \rmM^{\loc, (n)}$.
\end{proof}

\begin{remark}
    In fact one can be more precise than just saying that the morphism $\Disp_{\calG, \bmmu}^{(m', n')} \to \Disp_{\calG, \bmmu}^{(m, n)}$ is smooth.
    It factors as
    \[
        \Disp_{\calG, \bmmu}^{(m', n')} \to \Disp_{\calG, \bmmu}^{(m', n)} \to \Disp_{\calG, \bmmu}^{(m, n)}
    \]
    where the first morphism is an affine bundle and the second one is a gerbe banded by a unipotent group.
\end{remark}

\begin{lemma} \label{lem:restriction-smooth-dieudonne}
    Let $R' \to R$ be a surjection of complete Noetherian local $\calO_{\Ebrev}$-algebras with residue field $\Fpbar$.
    Suppose that we are given a Dieudonné $(\calG, \bmmu)$-display $(M, M_1, \Psi)$ over $R$, an $(m, n)$-truncated $(\calG, \bmmu)$-display $(\overline{M'}, \overline{M'_1}, \overline{\Psi'})$ over $R'$ and an isomorphism between the two induced $(m, n)$-truncated $(\calG, \bmmu)$-displays over $R$.
    Then $(M, M_1, \Psi)$ and $(\overline{M'}, \overline{M'_1}, \overline{\Psi'})$ admit a compatible lift to a Dieudonné $(\calG, \bmmu)$-display over $R'$.

    The statement is also true for $n = \onerdt$ when $R, R'$ are $\Fpbar$-algebras.
\end{lemma}

\begin{proof}
    We prove the statement in the case $n \neq \onerdt$ (that case then follows from \Cref{lem:restriction-smooth}).
    Using the smoothness of $\calG$ we can certainly find a compatible lift of $(M, M_1)$ and $(\overline{M'}, \overline{M'_1})$ to a Dieudonné $(\calG, \bmmu)$-pair $(M', M'_1)$ over $R'$.
    Then $\Psi$ and $\overline{\Psi'}$ together induce an isomorphism
    \[
        \roundbr[\big]{\Psi, \overline{\Psi'}} \colon \roundbr[\big]{\What(R) \times_{W_n(R)} W_n(R')} \otimes_{\What(R')} \widetilde{M'_1} \to \roundbr[\big]{\What(R) \times_{W_n(R)} W_n(R')} \otimes_{\What(R')} M'.
    \]
    Again using the smoothness of $\calG$ we can lift this isomorphism to an isomorphism $\Psi' \colon \widetilde{M'_1} \to M'$ to obtain the desired Dieudonné $(\calG, \bmmu)$-display $(M', M'_1, \Psi')$ over $R'$.
\end{proof}

The following Proposition is a variant of \cite[Lemma 3.1.17 and Lemma 3.2.13]{kisin-pappas} and will be needed later.

\begin{proposition} \label{prop:dieudonne-display-canonical-iso}
    Let $R$ be a complete Noetherian local ring with perfect residue field $k$ that is additionally $p$-torsionfree and satisfies $\frakm_R = \sqrt{pR}$, and write $W \coloneqq W(k)$ and $K_0 \coloneqq W[1/p]$.

    Let $(M_0, \Phi_0)$ be a tuple consisting of a finite free $W$-module (of rank $h$) and a morphism of $W$-modules $\Phi_0 \colon M_0^{\sigma} \to M_0$ such that $p M_0 \subseteq \im(\Phi_0)$ (note that this implies that $\Phi_0$ becomes an isomorphism after inverting $p$).

    Let $(M, \Phi)$ be a tuple consisting of a finite free $\What(R)$-module $M$ and a morphism of $\What(R)$-modules $\Phi \colon M^{\sigma} \to M$ that becomes an isomorphism after inverting $p$, and suppose we are given an identification $(M_0, \Phi_0) = W \otimes_{\What(R)} (M, \Phi)$.

    Then there exists a unique morphism of $\What(R)[1/p]$-modules
    \[
        \What(R)[1/p] \otimes_{K_0} M_0[1/p] \to M[1/p]
    \]
    that lifts the identity on $M_0[1/p]$ and is compatible with $\Phi_0$ and $\Phi$ in the sense that the diagram
    \[
    \begin{tikzcd}
        \What(R)[1/p] \otimes_{K_0} M_0[1/p]^{\sigma} \ar[r] \ar[d, "\Phi_0"]
        & M[1/p]^{\sigma} \ar[d, "\Phi"]
        \\
        \What(R)[1/p] \otimes_{K_0} M_0[1/p] \ar[r]
        & M[1/p]
    \end{tikzcd}
    \]
    is commutative.
    Furthermore, this morphism is in fact an isomorphism.

    Now suppose that $M_0$ and $M$ are additionally equipped with compatible $\calG$-structures such that $\Phi_0$ and $\Phi$ give isomorphisms of $\calG$-torsors after inverting $p$.
    Then also the above isomorphism is compatible with the $\calG$-structures on both sides.
\end{proposition}

\begin{proof}
    Choose a $\What(R)$-basis of $M$.
    Then the morphism $\Phi$ is given by a matrix
    \[
        A \in \GL_h(\What(R)[1/p]) \cap \M_h(\What(R))
    \]
    that we write as $A = A_0 + A'$ with $A_0 \in \GL_h(K_0) \cap \M_h(W)$ and $A' \in \M_h(\What(\frakm_R))$ (note that our assumptions imply that $p A_0^{-1} \in \M_h(W)$).
    Our task is now to show that there exists a unique matrix
    \[
        T \in \roundbr[\Big]{1 + \M_h \roundbr[\big]{\What(\frakm_R)[1/p]}} \quad \text{such that} \quad T = A \sigma(T) A_0^{-1}
    \]
    and that this matrix $T$ is furthermore contained in $\GL_h(\What(R)[1/p])$.
    \begin{itemize}
        \item
        For the uniqueness statement, suppose we are given $T$ and $T'$ as above.
        Then the difference
        \[
            U = T - T' \in \M_h(\What(\frakm_R)[1/p])
        \]
        again satisfies the equation $U = A \sigma(U) A_0^{-1}$ and thus
        \[
            U = A \dotsb \sigma^{m - 1}(A) \roundbr[\big]{p^{-m} \sigma^m(U)} \sigma^{m - 1}(p A_0^{-1}) \dotsb p A_0^{-1}
        \]
        for all $m \in \ZZ_{\geq 0}$.
        Now all the factors except for the middle one are valued in $\What(R)$ while the middle factor converges to $0$ for $m \to \infty$.
        Thus we have $U = 0$ as desired.

        \item
        For the existence statement, we set
        \[
            T_m \coloneqq A \dotsb \sigma^{m - 1}(A) \sigma^m(A' A_0^{-1}) \sigma^{m - 1}(A_0^{-1}) \dotsb A_0^{-1} \in \M_h \roundbr[\big]{\What(\frakm_R)[1/p]}.
        \]
        for $m \in \ZZ_{\geq 0}$ so that $T_m \to 0$ for $m \to \infty$ similarly to before.
        Note that we have the identities
        \[
            A A_0^{-1} = 1 + T_0 \quad \text{and} \quad A \sigma(T_m) A_0^{-1} = T_{m + 1}.
        \]
        Thus the matrix
        \[
            T \coloneqq 1 + \sum_{m = 0}^{\infty} T_m \in \roundbr[\Big]{1 + \M_h \roundbr[\big]{\What(\frakm_R)[1/p]}}
        \]
        satisfies the equation $T = A \sigma(T) A_0^{-1}$.

        \item
        It remains for us to show that we have $T \in \GL_h(\What(R)[1/p])$.
        As $1 - T \in \M_h(\What(\frakm_R)[1/p])$ we in fact have
        \[
            1 - \sigma^m(T) = \sigma^m(1 - T) \in \M_h \roundbr[\big]{\What(\frakm_R)}
        \]
        for $m \gg 0$ so that
        \[
            \sigma^m(T) \in \roundbr[\Big]{1 + \M_h \roundbr[\big]{\What(\frakm)}} \subseteq \GL_h \roundbr[\big]{\What(R)} \subseteq \GL_h \roundbr[\big]{\What(R)[1/p]}.
        \]
        But now we have
        \[
            T = A \dotsb \sigma^{m - 1}(A) \sigma^m(T) \sigma^{m - 1}(A_0^{-1}) \dotsb A_0^{-1}
        \]
        so that $T \in \GL_h(\What(R)[1/p])$ as well.
    \end{itemize}

    Now suppose that $M_0$ and $M$ are equipped with $\calG$-structures as described in the statement of the proposition.
    Fix a family of tensors $(s_{\alpha})_{\alpha}$ of tensors $s_{\alpha} \colon \Lambda^{\otimes \ell_{\alpha}} \to \Lambda^{\otimes \ell_{\alpha}}$ that cut out $\calG \subseteq \GL(\Lambda)$.
    Then the $\calG$-structures on $M_0$ and $M$ are given by families of tensors $(m_{0, \alpha})_{\alpha}$ and $(m_{\alpha})_{\alpha}$.
    There now certainly exists an isomorphism of $\What(R)$-modules $\What(R) \otimes_W M_0 \to M$ that lifts the identity of $M_0$ and is compatible with $\calG$-structures.
    Choose a $W$-basis of $M_0$ and use this isomorphism to obtain a $\What(R)$-basis of $M$.
    When we use this basis in the construction above, then every term in the series expansion
    \[
        T = 1 + \sum_{m = 0}^{\infty} T_m \colon \What(R)[1/p] \otimes_{K_0} M_0[1/p] \to M[1/p]
    \]
    is compatible with the tensors in the sense that the diagram
    \[
    \begin{tikzcd}
        \What(R)[1/p] \otimes_{K_0} M_0[1/p] \ar[r] \ar[d, "m_{0, \alpha}"]
        & M[1/p] \ar[d, "m_{\alpha}"]
        \\
        \What(R)[1/p] \otimes_{K_0} M_0[1/p] \ar[r]
        & M[1/p]
    \end{tikzcd}
    \]
    is commutative for all $\alpha$.
\end{proof}

\begin{corollary} \label{cor:g-structure-unique}
    Let $R$ be a complete Noetherian local $\calOE$-algebra with perfect residue field $k$ that is additionally $p$-torsionfree and satisfies $\frakm_R = \sqrt{pR}$.
    Let $(M_0, M_{0, 1}, \Psi_0)$ be a $(\calG, \bmmu)$-display over $k$ and let $(M, M_1, \Psi)$ be a deformation to a Dieudonné display over $R$.

    Then there exists at most one $\calG$-structure on $M$ that is compatible with the given $\calG$-structure on $M_0$ and makes $(M, M_1, \Psi)$ into a Dieudonné $(\calG, \bmmu)$-display over $R$.
\end{corollary}

\begin{proof}
    This is immediate from \Cref{prop:dieudonne-display-canonical-iso}.
\end{proof}

\subsection{\texorpdfstring{Grothendieck-Messing Theory for Dieudonné $(\calG, \bmmu)$-displays}{Grothendieck-Messing Theory for Dieudonné (G, mu)-displays}}

\begin{notation} \label{not:g-mu-grothendieck-messing}
    In this subsection $S \to R$ always denotes a surjection of complete Noetherian local $\calOE$-algebras with perfect residue field such that $S$ is $p$-torsionfree and satisfies $\frakm_S = \sqrt{pS}$ and such that the kernel $\fraka$ of $S \to R$ has nilpotent divided powers (that are automatically continuous).
\end{notation}

\begin{definition} \label{def:dieudonne-g-mu-pair-grothendieck-messing}
    A \emph{Dieudonné $(\calG, \bmmu)$-pair for $S/R$} is a tuple $(M, M_1)$ consisting of a finite free $\What(S)$-module $M$ that is equipped with a $\calG$-structure and a $\What(R)$-submodule $M_1 \subseteq \What(R) \otimes_{\What(S)} M$ such that $(\What(R) \otimes_{\What(S)} M, M_1)$ is a Dieudonné $(\calG, \bmmu)$-pair over $R$.

    A Dieudonné $(\calG, \bmmu)$-pair for $S/R$ (or over $R$) is called \emph{liftable (to $S$)} if it admits a lift to a Dieudonné $(\calG, \bmmu)$-pair over $S$.
    A Dieudonné $(\calG, \bmmu)$-display over $R$ is called \emph{liftable (to $S$)} if its underlying $(\calG, \bmmu)$-pair has this property.
\end{definition}

\begin{lemma} \label{lem:g-mu-m-1-grothendieck-messing}
    There are unique functors
    \[
    \begin{array}{r c l}
        \curlybr[\big]{\text{liftable Dieudonné $(\calG, \bmmu)$-pairs for $S/R$}}
        & \to
        & \curlybr[\big]{\text{$\calG$-torsors over $\What(S)$}}, \\[1mm]
        (M, M_1)
        & \mapsto
        & \widetilde{M_1}
    \end{array}
    \]
    together with natural isomorphisms $\widetilde{M_1}[1/p] \to M^{\sigma}[1/p]$ that are compatible with passing to Dieudonné pairs for $S/R$, see \Cref{lem:m-1-grothendieck-messing}.
    This construction is compatible with base change in $S/R$.
\end{lemma}

\begin{proof}
    As $\What(S)$ is $p$-torsionfree, given a liftable Dieudonné $(\calG, \bmmu)$-pair $(M, M_1)$ for $S/R$ there exists at most one $\calG$-structure on the finite free $\What(S)$-module $\widetilde{M_1}$ from \Cref{lem:m-1-grothendieck-messing} that is compatible with the $\calG$-structure on $M^{\sigma}$ under the isomorphism $\widetilde{M_1}[1/p] \to M^{\sigma}[1/p]$.
    The existence of such a $\calG$-structure follows immediately from the liftablility of $(M, M_1)$ to $S$ and \Cref{prop:g-mu-m-1-tilde}.
\end{proof}

\begin{definition}
    A \emph{liftable Dieudonné $(\calG, \bmmu)$-display for $S/R$} is a tuple $(M, M_1, \Psi)$ where $(M, M_1)$ is a liftable Dieudonné $(\calG, \bmmu)$-pair for $S/R$ and $\Psi \colon \widetilde{M_1} \to M$ is an isomorphism of $\calG$-torsors over $\What(S)$.
\end{definition}

\begin{theorem} \label{thm:g-mu-grothendieck-messing}
    The natural forgetful functor
    \[
        \curlybr[\Bigg]{\begin{gathered} \text{liftable Dieudonné $(\calG, \bmmu)$-displays} \\ \text{for $S/R$} \end{gathered}} \to \curlybr[\Bigg]{\begin{gathered} \text{liftable Dieudonné $(\calG, \bmmu)$-displays} \\ \text{over $R$} \end{gathered}}
    \]
    is an equivalence.
\end{theorem}

\begin{proof}
    The essential surjectivity of the functor is clear, so that it remains to check fully faithfulness.
    So suppose we are given two liftable Dieudonné $(\calG, \bmmu)$-displays $(M, M_1, \Psi)$ and $(M', M'_1, \Psi')$ for $S/R$ and an isomorphism between the induced Dieudonné $(\calG, \bmmu)$-displays over $R$.
    By \Cref{thm:grothendieck-messing} this isomorphism lifts uniquely to an isomorphism of Dieudonné displays $(M, M_1, \Psi) \to (M', M'_1, \Psi')$ for $S/R$ that automatically is compatible with $\calG$-structures by \Cref{prop:dieudonne-display-canonical-iso}.
\end{proof}

\begin{remark} \label{rmk:g-mu-grothendieck-messing}
    We expect that there exist natural functors
    \[
    \begin{array}{r c l}
        \curlybr[\big]{\text{liftable Dieudonné $(\calG, \bmmu)$-pairs for $S/R$}}
        & \to
        & \curlybr[\big]{\text{$\calG$-torsors over $\What(S)$}}, \\[1mm]
        (M, M_1)
        & \mapsto
        & \widetilde{M_1}
    \end{array}
    \]
    as in \Cref{lem:g-mu-m-1-grothendieck-messing}, but without the assumption that $S$ is $p$-torsionfree and $\frakm_S = \sqrt{pS}$ and that \Cref{thm:g-mu-grothendieck-messing} still holds true in that setting.
    \begin{itemize}
        \item
        The main obstruction to carrying this out seems to be the following.
        Given a Dieudonné $(\calG, \bmmu)$-pair $(M, M_1)$ for $S/R$ and two lifts $(M, M'_1)$ and $(M, M''_1)$ to Dieudonné $(\calG, \bmmu)$-pairs over $S$ we don't know if the natural isomorphism of finite free $\What(S)$-modules
        \[
            \widetilde{M'_1} \to \widetilde{M''_1}
        \]
        given by \Cref{lem:m-1-grothendieck-messing} is compatible with $\calG$-structures.

        \item
        Suppose that $\calG \subseteq \GL(\Lambda)$ is cut out by endomorphisms of $\Lambda$ and a homogeneous polarization.
        Then \Cref{lem:g-mu-m-1-grothendieck-messing} and \Cref{thm:g-mu-grothendieck-messing} are indeed true without the assumptions on $S$; the $\calG$-structure on $\widetilde{M_1}$ has a description similar to the one in \Cref{rmk:g-mu-m-1-tilde-pel}.
    \end{itemize}
\end{remark}

\subsection{\texorpdfstring{The universal deformation of a $(\calG, \bmmu)$-display}{The universal deformation of a (G, mu)-display}} \hfill\\ \label{subsec:universal-deformation-g-mu}
The goal of this subsection is to give a construction of a universal deformation of a $(\calG, \bmmu)$-display over a perfect field $k$ to a Dieudonné $(\calG, \bmmu)$-display, see \Cref{thm:universal-deformation-g-mu-display}, similarly to what was done in Subsection \ref{sec:universal-deformation-display}.
Some of our arguments are inspired by \cite[Subsection 3.2]{kisin-pappas}.

\begin{notation}
    Let $F/E$ be a complete discretely valued field with perfect residue field $k$ and write $W \coloneqq W(k)$.
    Let $(M_0, M_{0, 1}, \Psi_0)$ be a $(\calG, \bmmu)$-display over $k$.
    For simplicity we assume that there exists a trivialization $M_0 \to \Lambda_W$ (this is always possible after replacing $F$ by a finite unramified extension).
    
    We consider the deformation problem $\Def_{\calG}$ over $\calO_F$ given by
    \[
        \Def_{\calG} \colon R \mapsto \curlybr[\Bigg]{\begin{gathered} \text{deformations $(M, M_1, \Psi)$ of $(M_0, M_{0, 1}, \Psi_0)$ over $R$} \\ \text{as a Dieudonné $(\calG, \bmmu)$-display} \end{gathered}}.
    \]
    We also write $\Def_{\GL(\Lambda)}$ for the deformation problem that was called $\Def$ in Subsection \ref{sec:universal-deformation-display} and is given by
    \[
        \Def_{\GL(\Lambda)} \colon R \mapsto \curlybr[\Bigg]{\begin{gathered} \text{deformations $(M, M_1, \Psi)$ of $(M_0, M_{0, 1}, \Psi_0)$ over $R$} \\ \text{as a Dieudonné display} \end{gathered}}.
    \]
    Note that we have a natural morphism $\Def_{\calG} \to \Def_{\GL(\Lambda)}$.

    In this subsection $R$ and $R'$ always denote complete local Noetherian $\calO_F$-algebras with residue field $k$ and $K$ always denotes a finite extension of $F$ with uniformizer $\varpi_K$ and residue field $k_K$.
\end{notation}

\begin{construction} \label{con:universal-deformation-g-mu}
    Write
    \[
        R_{\GL(\Lambda)}, \qquad
        \roundbr[\big]{M^{\GL(\Lambda)}, M^{\GL(\Lambda)}_1} \quad \text{and} \quad
        \fraka_{\GL(\Lambda)}
    \]
    for the objects that were called $R^{\univ}$, $(M^{\univ}, M^{\univ}_1)$ and $\fraka$ respectively in Construction \ref{con:universal-deformation} and equip $\frakm_{R_{\GL(\Lambda)}}/\fraka_{\GL(\Lambda)} \subseteq R_{\GL(\Lambda)}/\fraka_{\GL(\Lambda)}$ with the trivial divided powers.
    Equip $M^{\GL(\Lambda)}$ with its natural $\calG$-structure (coming from the defining isomorphism $M^{\GL(\Lambda)} = \Lambda_{R_{\GL(\Lambda)}}$).

    Write $R_{\calG}$ for the completed local ring of $\rmM^{\loc}_{\calO_F}$ at the $k$-point corresponding to $M_{0, 1}/pM_0 \subseteq M_{0, k}$ so that $R_{\calG}$ is a quotient of $R_{\GL(\Lambda)}$.
    Write
    \[
        \roundbr[\big]{M^{\calG}, M^{\calG}_1} \coloneqq \roundbr[\big]{M^{\GL(\Lambda)}, M^{\GL(\Lambda)}_1}_{R_{\calG}}
    \]
    for the base change of the Dieudonné pair $(M^{\GL(\Lambda)}, M^{\GL(\Lambda)}_1)$ to $R_{\calG}$ and note that it is a Dieudonné $(\calG, \bmmu)$-pair that is a deformation of $(M_0, M_{0, 1})$.
    Finally write
    \[
        \fraka_{\calG} \coloneqq \frakm_{R_{\calG}}^2 + \frakm_{\Ebrev} R_{\calG}
    \]
    (or equivalently $\fraka_{\calG} \coloneqq \fraka_{\GL(\Lambda)} R_{\calG}$) and equip $\frakm_{R_{\calG}}/\fraka_{\calG} \subseteq R_{\calG}/\fraka_{\calG}$ with the trivial divided powers as well.
\end{construction}

We will now make the following assumption.

\begin{assumption} \label{hyp:grothendieck-messing}
    The natural isomorphism of finite free $\What(R_{\calG}/\fraka_{\calG})$-modules
    \[
        \What(R_{\calG}/\fraka_{\calG}) \otimes_{\What(R_{\calG})} \widetilde{M^{\calG}_1} \to \roundbr[\big]{\widetilde{M_{0, 1}}}_{\What(R_{\calG}/\fraka_{\calG})}
    \]
    given by \Cref{lem:m-1-grothendieck-messing} is compatible with $\calG$-structures.
\end{assumption}

\begin{remark}
    Assumption \ref{hyp:grothendieck-messing} is implicitly used in \cite[Subsection 3.2]{kisin-pappas}.
    We expect that it is always satisfied.

    \begin{itemize}
        \item
        Note that the assumption is a special case of the expectation formulated in \Cref{rmk:g-mu-grothendieck-messing}.
        In particular it certainly holds true when $\calG \subseteq \GL(\Lambda)$ is cut out by endomorphisms and a homogeneous polarization.

        \item
        The assumption would also follow if every $k[\varepsilon]$-point of $\Spf(R_{\calG})$ could be lifted to an $\calO_K[\varepsilon]$-point for some $K$, see \Cref{lem:g-mu-m-1-grothendieck-messing}.
        In fact it would even be sufficient to only lift a generating set of tangent vectors.

        \item
        There is work in progress by Pappas and Zhou that actually proves the assumption in general by studying the geometry of the special fiber of $\rmM^{\loc}$.
    \end{itemize}
\end{remark}

\begin{construction}[continues=con:universal-deformation-g-mu] \label{con:universal-deformation-g-mu-2}
    The composition
    \[
        \What(R_{\calG}/\fraka_{\calG}) \otimes_{\What(R_{\calG})} \widetilde{M^{\calG}_1} \to \roundbr[\big]{\widetilde{M_{0, 1}}}_{\What(R_{\calG}/\fraka_{\calG})} \xrightarrow{\Psi_0} M_{0, \What(R_{\calG}/\fraka_{\calG})} \to \What(R_{\calG}/\fraka_{\calG}) \otimes_{\What(R_{\calG})} M^{\calG}
    \]
    now is an isomorphism of $\calG$-torsors over $\What(R_{\calG}/\fraka_{\calG})$.
    We choose a lift of this isomorphism to an isomorphism of $\calG$-torsors over $\What(R_{\calG})$
    \[
        \Psi^{\calG} \colon \widetilde{M^{\calG}_1} \to M^{\calG}
    \]
    so that we obtain a deformation
    \[
        \roundbr[\big]{M^{\calG}, M^{\calG}_1, \Psi^{\calG}} \in \Def_{\calG}(R_{\calG}).
    \]
    Now $\Psi^{\calG}$ and the composition
    \begin{align*}
        \What(R_{\GL(\Lambda)}/\fraka_{\GL(\Lambda)}) &\otimes_{\What(R_{\GL(\Lambda)})} \widetilde{M^{\GL(\Lambda)}_1} \to \roundbr[\big]{\widetilde{M_{0, 1}}}_{\What(R_{\GL(\Lambda)}/\fraka_{\GL(\Lambda)})} \xrightarrow{\Psi_0} \\
        &\xrightarrow{\Psi_0} M_{0, \What(R_{\GL(\Lambda)}/\fraka_{\GL(\Lambda)})} \to \What(R_{\GL(\Lambda)}/\fraka_{\GL(\Lambda)}) \otimes_{\What(R_{\GL(\Lambda)})} M^{\GL(\Lambda)}
    \end{align*}
    together induce an isomorphism of finite free $\What(R_{\calG} \times_{R^{\calG}/\fraka_{\calG}} R_{\GL(\Lambda)}/\fraka_{\GL(\Lambda)})$-modules
    \[
        \What \roundbr[\big]{R_{\calG} \times_{R^{\calG}/\fraka_{\calG}} R_{\GL(\Lambda)}/\fraka_{\GL(\Lambda)}} \otimes_{\What(R^{\GL(\Lambda)})} \widetilde{M^{\GL(\Lambda)}_1} \to \What \roundbr[\big]{R_{\calG} \times_{R^{\calG}/\fraka_{\calG}} R_{\GL(\Lambda)}/\fraka_{\GL(\Lambda)}} \otimes_{\What(R^{\GL(\Lambda)})} M^{\GL(\Lambda)}.
    \]
    We further choose a lift of this isomorphism to an isomorphism of finite free $R_{\GL(\Lambda)}$-modules
    \[
        \Psi^{\GL(\Lambda)} \colon \widetilde{M^{\GL(\Lambda)}_1} \to M^{\GL(\Lambda)}
    \]
    so that we obtain a deformation
    \[
        \roundbr[\big]{M^{\GL(\Lambda)}, M^{\GL(\Lambda)}_1, \Psi^{\GL(\Lambda)}} \in \Def_{\GL(\Lambda)}(R_{\GL(\Lambda)})
    \]
    that is universal by \Cref{thm:universal-deformation-display}.

    In summary we have now constructed a commutative diagram
    \[
    \begin{tikzcd}
        \Spf(R_{\calG}) \ar[r] \ar[d]
        & \Spf \roundbr[\big]{R_{\GL(\Lambda)}} \ar[d]
        \\
        \Def_{\calG} \ar[r]
        & \Def_{\GL(\Lambda)}
    \end{tikzcd}
    \]
    where the right vertical morphism is an isomorphism.
    In the following we will abuse notation by identifying $\Spf(R_{\GL(\Lambda)})$ with $\Def_{\GL(\Lambda)}$ and viewing $\Spf(R_{\calG})$ as a sub-deformation problem of $\Def_{\calG}$.

    The goal is now to show that
    \[
        (M^{\calG}, M^{\calG}_1, \Psi^{\calG}) \in \Def_{\calG}(R_{\calG})
    \]
    is a universal deformation, i.e.\ that $\Def_{\calG} = \Spf(R_{\calG})$.
\end{construction}

\begin{lemma} \label{lem:def-g-fiber-products}
    Let
    \[
    \begin{tikzcd}
        R''' \ar[r] \ar[d]
        & R'' \ar[d]
        \\
        R' \ar[r]
        & R
    \end{tikzcd}
    \]
    be a fiber product diagram of complete Noetherian local $\calO_{\Ebrev}$-algebras with residue field $\Fpbar$ such that the following conditions are satisfied.
    \begin{itemize}
        \item
        $R'$ and $R'''$ are $p$-torsionfree and $\frakm_{R'} = \sqrt{pR'}$ and $\frakm_{R'''} = \sqrt{pR'''}$.

        \item
        The ideal $\ker(R' \to R) \subseteq R'$ has nilpotent divided powers (this implies that $\ker(R''' \to R'') \subseteq R'''$ has nilpotent divided powers as well).
    \end{itemize}
    Then the map
    \[
        \Def_{\calG}(R''') \to \Def_{\calG}(R') \times_{\Def_{\calG}(R)} \Def_{\calG}(R'')
    \]
    is a bijection.
\end{lemma}

\begin{proof}
    This follows from the argument that was given in the second bullet point of the proof of \Cref{prop:deformation-problem-display-pro-representable}, using \Cref{thm:g-mu-grothendieck-messing} instead of \Cref{thm:grothendieck-messing}.
\end{proof}

\begin{remark}
    The assumptions in \Cref{lem:def-g-fiber-products} are in particular satisfied for the fiber products
    \begin{gather*}
        \calO_K[\varepsilon] \times_{\calO_K} \calO_K[t]/t^{\ell}, \qquad
        \calO_K[\varepsilon_1, \varepsilon_2] \times_{\calO_K} \calO_K[t]/t^{\ell}, \\
        \calO_K[t]/t^{\ell + 1} \times_{\calO_K[t]/t^{\ell}} \calO_K[t]/t^{\ell + 1},
    \end{gather*}
    for $\ell \in \ZZ_{> 0}$.

    Thus we have well-defined tangent spaces
    \[
        T_x \Def_{\calG}
    \]
    at points $x \in \Def_{\calG}(\calO_K)$, and given a point $y \in \Def_{\calG}(\calO_K[t]/t^{\ell})$ with image $x \in \Def_{\calG}(\calO_K)$, the fiber
    \[
        \fib_y \roundbr[\Big]{\Def_{\calG} \roundbr[\big]{\calO_K[t]/t^{\ell + 1}}}
    \]
    naturally has the structure of a $T_x \Def_{\calG}$-pseudotorsor (see for example \cite{schlessinger}).
\end{remark}

\begin{lemma} \label{lem:def-g-tangent-space}
    Let $x \in \Spf(R_{\calG})(\calO_K)$.
    Then we have
    \[
        T_x \Spf(R_{\calG}) = T_x \Def_{\calG}.
    \]
\end{lemma}

\begin{proof}
    \Cref{thm:g-mu-grothendieck-messing} produces an isomorphism
    \[
        T_x \Spf(R_{\calG}) = T_x \rmM^{\loc} \to T_x \Def_{\calG}
    \]
    (that does not need to agree with the inclusion $T_x \Spf(R_{\calG}) \subseteq T_x \Def_{\calG}$).
    This implies that $T_x \Def_{\calG}$ is a finite free $\calO_K$-module of the same rank as $T_x \Spf(R_{\calG})$.
    Now the composition
    \[
        T_x \Spf(R_{\calG}) \subseteq T_x \Def_{\calG} \to T_x \Def_{\GL(\Lambda)} = T_x \Spf(R_{\GL(\Lambda)})
    \]
    is a split monomorphism of $\calO_K$-modules because $R_{\GL(\Lambda)} \to R_{\calG}$ is surjective.
    But this now implies that $T_x \Spf(R_{\calG}) = T_x \Def_{\calG}$ as desired.
\end{proof}

\begin{lemma} \label{lem:def-g-spreading-out}
    Let $(M, M_1, \Psi) \in \Def_{\calG}(\calO_K \doublesquarebr{t})$ such that
    \[
        x = (M, M_1, \Psi)_{\calO_K, t \mapsto 0} \in \Spf(R_{\calG})(\calO_K).
    \]
    Then we have
    \[
        (M, M_1, \Psi) \in \Spf(R_{\calG}) \roundbr[\big]{\calO_K \doublesquarebr{t}}.
    \]
\end{lemma}

\begin{proof}
    We inductively show that
    \[
        (M, M_1, \Psi)_{\calO_K[t]/t^{\ell}} \in \Spf(R_{\calG}) \roundbr[\big]{\calO_K[t]/t^{\ell}}
    \]
    for all $\ell \in \ZZ_{> 0}$.
    
    So assume that this holds for a fixed $\ell$.
    Setting $y = (M, M_1, \Psi)_{\calO_K[t]/t^{\ell}}$ we then have
    \[
        (M, M_1, \Psi)_{\calO_K[t]/t^{\ell + 1}} \in \fib_y \roundbr[\Big]{\Def_{\calG} \roundbr[\big]{\calO_K[t]/t^{\ell + 1}}}
    \]
    so that it suffices to show that the map
    \[
        \fib_y \roundbr[\Big]{\Spf(R_{\calG}) \roundbr[\big]{\calO_K[t]/t^{\ell + 1}}} \to \fib_y \roundbr[\Big]{\Def_{\calG} \roundbr[\big]{\calO_K[t]/t^{\ell + 1}}}
    \]
    is a bijection.
    As it is in fact a morphism of $(T_x \Spf(R_{\calG}) = T_x \Def_{\calG})$-pseudotorsors (see \Cref{lem:def-g-tangent-space}) it is furthermore sufficient to show that $\fib_y (\Spf(R_{\calG})(\calO_K[t]/t^{\ell + 1}))$ is non-empty.

    Now the assumption $y \in \Spf(R_{\calG})(\calO_K[t]/t^{\ell})$ provides us with a trivialization
    \[
        \What \roundbr[\big]{\calO_K[t]/t^{\ell}} \otimes_{\What(\calO_K \doublesquarebr{t})} M \to \What \roundbr[\big]{\calO_K[t]/t^{\ell}} \otimes_{y, \What(R_{\calG})} M^{\calG} \to \Lambda_{\What(\calO_K/t^{\ell})}
    \]
    such that $y \in \Spf(R_{\calG})(\calO_K[t]/t^{\ell})$ corresponds to the direct summand
    \begin{gather*}
        \What \roundbr[\big]{\calO_K[t]/t^{\ell}} \otimes_{\What(\calO_K \doublesquarebr{t})} \roundbr[\big]{M_1/\Ihat_{\calO_K \doublesquarebr{t}} M} \\
        \subseteq \What \roundbr[\big]{\calO_K[t]/t^{\ell}} \otimes_{\What(\calO_K \doublesquarebr{t})} \roundbr[\big]{M/\Ihat_{\calO_K \doublesquarebr{t}} M} \cong \Lambda_{\calO_K[t]/t^{\ell}}.
    \end{gather*}
    Lifting this trivialization to one of
    \[
        \What \roundbr[\big]{\calO_K[t]/t^{\ell + 1}} \otimes_{\What(\calO_K \doublesquarebr{t})} M
    \]
    then gives rise to a lift $y' \in \Spf(R_{\calG})(\calO_K[t]/t^{\ell + 1})$ of $y$ corresponding to
    \begin{gather*}
        \What \roundbr[\big]{\calO_K[t]/t^{\ell + 1}} \otimes_{\What(\calO_K \doublesquarebr{t})} \roundbr[\big]{M_1/\Ihat_{\calO_K \doublesquarebr{t}} M} \\
        \subseteq \What \roundbr[\big]{\calO_K[t]/t^{\ell + 1}} \otimes_{\What(\calO_K \doublesquarebr{t})} \roundbr[\big]{M/\Ihat_{\calO_K \doublesquarebr{t}} M} \cong \Lambda_{\calO_K[t]/t^{\ell + 1}}
    \end{gather*}
    so that $\fib_y (\Spf(R_{\calG})(\calO_K[t]/t^{\ell + 1}))$ is non-empty as desired.
\end{proof}

\begin{lemma} \label{lem:def-g-connecting-points}
    Let $(M', M'_1, \Psi'), (M'', M''_1, \Psi'') \in \Def_{\calG}(\calO_K)$ such that there exists an isomorphism of Dieudonné $(\calG, \bmmu)$-pairs over $\calO_K$
    \[
        (M', M'_1) \to (M'', M''_1)
    \]
    inducing the identity on $(M_0, M_{0, 1})_{k_K}$.
    Then there exists
    \[
        (M, M_1, \Psi) \in \Def_{\calG} \roundbr[\big]{\calO_K \doublesquarebr{t}}
    \]
    that specializes to $(M', M'_1, \Psi')$ respectively $(M'', M''_1, \Psi'')$ under the morphism
    \[
        \calO_K \doublesquarebr{t} \to \calO_K, \qquad t \mapsto \varpi_K \quad \text{respectively} \quad t \mapsto 0.
    \]
\end{lemma}

\begin{proof}
    Choose an isomorphism $(M', M'_1) \cong (M'', M''_1)$ as above and define $(M, M_1)$ as the base change of any of these two Dieudonné $(\calG, \bmmu)$-pairs along $\calO_K \to \calO_K \doublesquarebr{t}$.
    Now consider the surjective morphism
    \[
        \calO_K \doublesquarebr{t} \to \calO_K \times_{k_K} \calO_K, \qquad t \mapsto (\varpi_K, 0).
    \]
    Then the isomorphisms $\Psi'$ and $\Psi''$ induce an isomorphism
    \[
        (\Psi', \Psi'') \colon (\calO_K \times_{k_K} \calO_K) \otimes_{\calO_K \doublesquarebr{t}} \widetilde{M_1} \to (\calO_K \times_{k_K} \calO_K) \otimes_{\calO_K \doublesquarebr{t}} M
    \]
    that we can lift to an isomorphism $\Psi \colon \widetilde{M_1} \to M$ to obtain an object
    \[
        (M, M_1, \Psi) \in \Def_{\calG} \roundbr[\big]{\calO_K \doublesquarebr{x}}
    \]
    with the desired properties.
\end{proof}

\begin{lemma} \label{lem:def-g-lift-reduced-p-torsionfree}
    Let $(M, M_1, \Psi) \in \Def_{\calG}(R)$.
    Then there exists a morphism $R' \to R$ with $R'$ reduced and $p$-torsionfree and a preimage
    \[
        (M', M'_1, \Psi') \in \Def_{\calG}(R')
    \]
    of $(M, M_1, \Psi)$.
\end{lemma}

\begin{proof}
    Using flatness and formal reducedness of $\rmM^{\loc}$ we can find $R' \to R$ with $R'$ reduced and $p$-torsionfree and a Dieudonné $(\calG, \bmmu)$-pair $(M', M'_1)$ over $R'$ together with an isomorphism $(M', M'_1)_R \cong (M, M_1)$.
    Replacing $R'$ with $R' \doublesquarebr{x_1, \dotsc, x_{\ell}}$ for a suitable $\ell \in \ZZ_{> 0}$ if necessary, we may even assume that $R' \to R$ is surjective.
    Then $\Psi$ can be lifted to an isomorphism $\Psi' \colon \widetilde{M'_1} \to M'$ so that we obtain our desired lift.
\end{proof}

\begin{theorem} \label{thm:universal-deformation-g-mu-display}
    Assume that Assumption \ref{hyp:grothendieck-messing} is satisfied.
    Then the deformation $(M^{\calG}, M^{\calG}_1, \Psi^{\calG}) \in \Def_{\calG}(R_{\calG})$ from Construction \ref{con:universal-deformation-g-mu} is universal.
\end{theorem}

\begin{proof}
    We first show that
    \[
        \Spf(R_{\calG})(\calO_K) = \Def_{\calG}(\calO_K).
    \]
    Given $(M', M'_1, \Psi') \in \Def_{\calG}(\calO_K)$ we can choose a trivialization $M' \to \Lambda_{\What(\calO_K)}$ lifting the identity on $M_0$ and then find a (unique) deformation
    \[
        (M'', M''_1, \Psi'') \in \Spf(R_{\calG})(\calO_K)
    \]
    such that the composition $M' \to \Lambda_{\What(\calO_K)} \to M''$ defines an isomorphism of Dieudonné $(\calG, \bmmu)$-pairs over $\calO_K$
    \[
        (M', M'_1) \to (M'', M''_1).
    \]
    Applying \Cref{lem:def-g-connecting-points} we obtain a deformation
    \[
        (M, M_1, \Psi) \in \Def_{\calG} \roundbr[\big]{\calO_K \doublesquarebr{t}}
    \]
    that specializes to $(M', M'_1, \Psi')$ and $(M'', M''_1, \Psi'')$ under $t \mapsto \varpi_K$ and $t \mapsto 0$ respectively.
    By \Cref{lem:def-g-spreading-out} we then have $(M, M_1, \Psi) \in \Spf(R_{\calG})(\calO_K \doublesquarebr{t})$, so that in particular
    \[
        (M', M'_1, \Psi') \in \Spf(R_{\calG})(\calO_K)
    \]
    as desired.

    Next we claim that the morphism
    \[
        \Def_{\calG} \to \Def_{\GL(\Lambda)} = \Spf(R_{\GL(\Lambda)})
    \]
    has image inside $\Spf(R_{\calG})$ so that we obtain a retraction $r \colon \Def_{\calG} \to \Spf(R_{\calG})$ of the inclusion.
    By \Cref{lem:def-g-lift-reduced-p-torsionfree} it suffices to show that the map
    \[
        \Def_{\calG}(R) \to \Spf(R_{\GL(\Lambda)})(R)
    \]
    has image inside $\Spf(R_{\calG})(R)$ for $R$ reduced and $p$-torsionfree.
    By noting that such $R$ always injects into a product of $\calO_K$'s we are further reduced to the case $R = \calO_K$ where we have already established the claim.

    Now let $x = (M, M_1, \Psi) \in \Def_{\calG}(R)$.
    Then the point $r(x) \in \Spf(R_{\calG})(R)$ gives rise to a second $\calG$-structure on the finite free $\What(R)$-module $M$ that also makes $(M, M_1, \Psi)$ into a Dieudonné $(\calG, \bmmu)$-display that is a deformation of $(M_0, M_{0, 1}, \Psi_0)$.
    We have to show that $r(x) = x$, which is the same as showing that both $\calG$-structures agree.
    \Cref{lem:def-g-lift-reduced-p-torsionfree} once more allows us to reduce to the case $R = \calO_K$ where this follows from \Cref{cor:g-structure-unique}.
\end{proof}

\begin{remark} \label{rmk:locally-universal}
    Let us recall the notions of \emph{sections rigid in the first order} and \emph{locally universal Dieudonné $(\calG, \bmmu)$-displays} given by Pappas in \cite[Definitions 4.5.8 and 4.5.10]{pappas}.
    \begin{itemize}
        \item
        Let $(M, M_1, \Psi) \in \Def_{\calG}(R)$, set $\fraka_R \coloneqq \frakm_R^2 + \frakm_{\Ebrev} R \subseteq R$ and equip $\frakm_R/\fraka_R \subseteq R/\fraka_R$ with the trivial divided powers.
        A trivialization $M \cong \Lambda_{\What(R)}$ (compatible with the given trivialization of $M_0$) is called \emph{rigid in the first order} if the diagram
        \[
        \begin{tikzcd}
            \What(R/\fraka_R) \otimes_{\What(R)} \widetilde{M_1} \ar[r, "\Psi"] \ar[d]
            & \What(R/\fraka_R) \otimes_{\What(R)} M \ar[d]
            \\
            \roundbr[\big]{\widetilde{M_{0, 1}}}_{\What(R/\fraka_R)} \ar[r, "\Psi_0"]
            & M_{0, \What(R/\fraka_R)}
        \end{tikzcd}
        \]
        is commutative.
        Here the vertical isomorphisms are induced by the trivialization and \Cref{lem:m-1-grothendieck-messing} (and both depend on the choice of trivialization).

        Trivializations that are rigid in the first order always exist (the standard trivialization of $M^{\calG}$ is rigid in the first order essentially by definition of $\Psi^{\calG}$ and this is the universal case).
        
        \item
        A deformation $(M, M_1, \Psi) \in \Def_{\calG}(R)$ is called \emph{locally universal} if there exists a trivialization $M \cong \What(R) \otimes_{\Zp} \Lambda$ that is rigid of the first order and such that the morphism $R_{\calG} \to R$ induced by the filtration $M_1$ is an isomorphism.
    \end{itemize}
    From Construction \ref{con:universal-deformation-g-mu} and \Cref{thm:universal-deformation-g-mu-display} it follows that a given deformation $(M, M_1, \Psi) \in \Def_{\calG}(R)$ is locally universal if and only if it is a universal deformation, i.e.\ if the corresponding morphism $\Spf(R) \to \Def_{\calG}$ is an isomorphism.
\end{remark}

\subsection{Comparison to restricted local shtukas} \hfill\\ \label{subsec:comp-local-shtukas}
In this subsection we compare our ($(m, n)$-truncated) $(\calG, \mu)$-displays with the ($(m, n)$-restricted) local shtukas studied by Xiao and Zhu in \cite[Section 5]{xiao-zhu} and by Shen, Yu and Zhang in \cite[Section 4]{shen-yu-zhang}.

\begin{notation}
    Write $\Fl = \bbL G/\bbL^+ \calG$ for the \emph{Witt vector affine flag variety} for $\calG$.
    We have the \emph{$\bmmu$-admissible locus} $\calA \subseteq \Fl_{\Fq}$ that is the union of those $\bbL^+ \calG$-orbits corresponding to elements in the (finite) \emph{$\bmmu$-admissible set}
    \[
        \Adm(\bmmu)_{\calG} \subseteq \calG(\Zpbrev) \backslash G(\Qpbrev) / \calG(\Zpbrev)
    \]
    (see \cite[Definition 3.11]{anschuetz-gleason-lourenco-richarz} and \cite{he-rapoport} for more details).
    Recall that the $\bbL^+ \calG$-action on $\calA$ factors through $\calG_{\Fp}^{\pf}$ and that the perfection of the special fiber of $\rmM^{\loc}$ naturally identifies with $\calA$.
    We now also set
    \[
        \calA^{(\infty)} \coloneqq \calA \times_{\Fl_{\Fq}} (\bbL G)_{\Fq}
    \]
    and equip $\calA^{(\infty)}$ with the action
    \[
        \roundbr[\big]{\bbL^+ \calG \times \bbL^+ \calG} \times \calA^{(\infty)} \to \calA^{(\infty)}, \qquad
        \roundbr[\big]{(k_1, k_2), g} \mapsto k_2 \cdot g \cdot \sigma^{-1}(k_1)^{-1}.
    \]
    This action makes $\calA^{(\infty)} \to \calA$ into an $L^+ \calG$-equivariant $L^+ \calG$-torsor.
    We write
    \[
        \calA^{(n)} \to \calA
    \]
    for its reduction to an $\bbL^{(n)} \calG$-torsor.
    The $\bbL^+ \calG$-equivariant structure on $\calA^{(n)} \to \calA$ then factors through $\bbL^{(m)} \calG$ by the argument in \cite[Lemma 4.2.2]{shen-yu-zhang}.
\end{notation}

With this setup, we can now give the definition of an ($(m, n)$-restricted) local shtuka.

\begin{definition}[Xiao-Zhu, Shen-Yu-Zhang]
    We define the \emph{stack of local shtukas for $(\calG, \bmmu)$} as the quotient stack
    \[
        \Sht^{\loc}_{\calG, \mu} \coloneqq \squarebr[\big]{(\bbL^+ \calG)_{\Delta} \backslash \calA_{\calG, \bmmu}^{(\infty)}}
    \]
    (where the subscript $\Delta$ indicates that we take the quotient by the diagonal action as usual).
    Similarly we also define the \emph{stack of $(m, n)$-restricted local shtukas for $(\calG, \mu)$} as
    \[
        \Sht^{\loc, (m, n)}_{\calG, \mu} \coloneqq \squarebr[\big]{(\bbL^{(m)} \calG)_{\Delta} \backslash \calA_{\calG, \bmmu}^{(n)}}.
    \]
\end{definition}

\begin{remark}
    We note that our normalization differs from the one in \cite{shen-yu-zhang}.
    This change actually also gives a slightly different notion of restricted shtukas.
\end{remark}

\begin{proposition} \label{prop:comparison-local-shtukas}
    There is a natural $(\bbL^+ \calG \times \bbL^+ \calG)$-equivariant isomorphism
    \[
        \calA^{(\infty)} \to \roundbr[\big]{\rmM^{\loc, (\infty)}}_{\Fq}^{\pf}
    \]
    that is compatible with the identification $\calA = (\rmM^{\loc})_{\Fq}^{\pf}$.
    Consequently we obtain equivalences
    \[
        \roundbr[\big]{\Disp_{\calG, \mu}}_{\Fq}^{\pf} \cong \Sht^{\loc}_{\calG, \mu} \quad \text{and} \quad \roundbr[\big]{\Disp_{\calG, \mu}^{(m, n)}}_{\Fq}^{\pf} \cong \Sht^{\loc, (m, n)}_{\calG, \mu}.
    \]
\end{proposition}

\begin{proof}
    Write $\calA_{\Lambda, d}$, $\calA_{\Lambda, d}^{(\infty)}$, $\rmM^{\loc}_{\Lambda, d}$ and $\rmM^{\loc, (\infty)}_{\Lambda, d}$ for the objects analogous to $\calA$, $\calA^{(\infty)}$, $\rmM^{\loc}$ and $\rmM^{\loc, (\infty)}$ when $(\calG, \bmmu)$ is replaced by $(\GL(\Lambda), \bmmu_d)$.
    For a $p$-complete ring $R$ we then make the identifications
    \[
        \rmM^{\loc}_{\Lambda, d}(R) \cong \set[\big]{M_1 \subseteq \Lambda_{W(R)}}{\text{$(\Lambda_{W(R)}, M_1)$ is a pair of type $(h, d)$}}
    \]
    and
    \[
        \rmM^{\loc, (\infty)}_{\Lambda, d}(R) \cong \set[\big]{(M_1, \xi)}{\text{$M_1 \in \rmM^{\loc}_{\Lambda, d}(R)$, $\xi \colon \widetilde{M_1} \to \Lambda_{W(R)}$ is an isomorphism}}.
    \]
    Then the identification $\calA_{\Lambda, d} = (\rmM^{\loc}_{\Lambda, d})_{\Fp}^{\pf}$ is explicitly given by
    \[
        \calA_{\Lambda, d} \to \roundbr[\big]{\rmM^{\loc}_{\Lambda, d}}_{\Fp}^{\pf}, \qquad g \cdot \bbL^+ \GL(\Lambda) \mapsto M_1^g \coloneqq p g \Lambda_{W(R)} \subseteq \Lambda_{W(R)}
    \]
    and we can define an isomorphism
    \[
        \calA_{\Lambda, d}^{(\infty)} \to \roundbr[\big]{\rmM^{\loc, (\infty)}_{\Lambda, d}}_{\Fp}^{\pf}, \qquad g \mapsto \roundbr[\big]{M_1^g, \; \xi^g \colon \widetilde{M_1^g} \cong p \sigma(g) \Lambda_{W(R)} \xrightarrow{p^{-1} \sigma(g)^{-1}} \Lambda_{W(R)}}
    \]
    that has all the required properties.
    We now claim that this isomorphism restricts to an isomorphism
    \[
        \calA^{(\infty)} \to \roundbr[\big]{\rmM^{\loc, (\infty)}}_{\Fq}^{\pf}.
    \]
    This amounts to showing that for a perfect $\Fq$-algebra $R$ and $g \in \calA^{(\infty)}(R)$ the isomorphism
    \[
        \xi^g \colon \widetilde{M^g_1} \to \Lambda_{W(R)}
    \]
    preserves $\calG$-structures.
    As the $\calG$-structure on $\widetilde{M^g_1}$ is characterized by the identification $\widetilde{M^g_1}[1/p] \cong \Lambda_{W(R)}^{\sigma}[1/p]$ this boils down to noting that $p^{-1} \sigma(g)^{-1} \in G(W(R)[1/p])$.
\end{proof}

\subsection{\texorpdfstring{Comparison to $(\calG, \mu)$-displays in the sense of Bültel-Pappas}{Comparison to (G, mu)-displays in the sense of Bültel-Pappas}} \hfill\\ \label{subsec:comp-bueltel-pappas}
In this subsection we compare our $(\calG, \bmmu)$-displays with the $(\calG, \mu)$-displays defined by Bültel and Pappas in \cite{bueltel-pappas}.

\begin{notation} \label{not:bueltel-pappas}
    Suppose that $\calG$ is reductive and fix a representative $\mu$ of $\bmmu$ that is defined over $\calO_F$ for some finite unramified extension $F/\Qp$.
    Set
    \[
        H_{\mu} \coloneqq P_{\mu^{-1}} \times_{\calG_{\calO_F}} (L^+ \calG)_{\calO_F},
    \]
    where $P_{\mu^{-1}} \subseteq \calG_{\calO_F}$ denotes the parabolic subgroups associated to $\mu^{-1}$ (our normalizations are different from the ones in \cite{bueltel-pappas} as we are normalizing our display functors contravariantly, compare \Cref{thm:p-div-groups-disps}).
    There exists a morphism of $\calO_F$-group schemes
    \[
        \varphi \colon H_{\mu} \to (L^+ \calG)_{\calO_F}
    \]
    that is characterized by the property
    \[
        \varphi(h) = \sigma \roundbr[\big]{\mu(p)^{-1} \cdot h \cdot \mu(p)} \in G \roundbr[\big]{W(R)[1/p]}
    \]
    for all $h \in H_{\mu}(R)$ (see \cite[Proposition 3.1.2]{bueltel-pappas}).

    Write $\Lambda_{\calO_F} = \Lambda_{-1} \oplus \Lambda_0$ for the weight decomposition induced by $\mu$ and define
    \[
        M^{\std}_1 \subseteq \Lambda_{W(\calO_F)}
    \]
    to be the preimage of $\Lambda_{-1}$ under $\Lambda_{W(\calO_F)} \to \Lambda_{\calO_F}$ so that $(\Lambda_{W(\calO_F)}, M^{\std}_1)$ is a $(\calG, \bmmu)$-pair over $\calO_F$.
    Note that $H_{\mu}$ naturally is the automorphism group scheme of $(\Lambda_{W(\calO_F)}, M^{\std}_1)$.

    In this subsection $R$ always denotes a $p$-complete $\calO_F$-algebra.
\end{notation}

Let us recall the definition of a $(\calG, \mu)$-display in the sense of Bültel-Pappas.

\begin{definition}
    A \emph{$(\calG, \mu)$-display over $R$ in the sense of Bültel-Pappas} is a tuple $(\calP, \calQ, \Psi)$ where $\calP$ is an $L^+ \calG$-torsor over $R$, $\calQ$ is a choice of reduction of $\calP$ to an $H_{\mu}$-torsor and
    \[
        \Psi \colon L^+ \calG \times^{\varphi, H_{\mu}} \calQ \to \calP
    \]
    is an isomorphism of $L^+ \calG$-torsors over $R$.
\end{definition}

\begin{proposition}
    There is a natural equivalence
    \[
    \begin{array}{r c l}
        \Pair_{\calG, \bmmu, \calO_F}
        &\to
        &\curlybr[\big]{\text{$H_{\mu}$-torsors}}, \\[2mm]
        (M, M_1)
        &\mapsto
        &\iIsom \roundbr[\big]{(M, M_1), (\Lambda_{W(\calO_F)}, M^{\std}_1)}.
    \end{array}
    \]
    Under this equivalence the construction $(M, M_1) \mapsto \widetilde{M_1}$ naturally identifies with $\calQ \mapsto L^+ \calG \times^{\varphi, H_{\mu}} \calQ$.
    Consequently we obtain a natural equivalence
    \[
        \Disp_{\calG, \bmmu, \calO_F} \cong \curlybr[\big]{\text{$(\calG, \mu)$-displays in the sense of Bültel-Pappas}}.
    \]
\end{proposition}

\begin{proof}
    The first claim is immediate because $\calG$ acts transitively on $\rmM^{\loc}$ so that every $(\calG, \bmmu)$-pair is locally isomorphic to $(\Lambda_{W(\calO_F)}, M^{\std}_1)$.

    For the second claim we note that the element $p \cdot \sigma(\mu(p)) \in G (W(\calO_F)[1/p])$ induces an isomorphism $\Lambda_{W(\calO_F)} \to \widetilde{M^{\std}_1}$.
    If we denote the inverse of this isomorphism by $\xi$ then $\varphi$ can be described as the morphism
    \[
        h \mapsto \roundbr[\big]{\Lambda_{W(R)} \xrightarrow{\xi^{-1}} \roundbr[\big]{\widetilde{M^{\std}_1}}_{W(R)} \xrightarrow{\widetilde{h}} \roundbr[\big]{\widetilde{M^{\std}_1}}_{W(R)} \xrightarrow{\xi} \Lambda_{W(R)}}.
    \]
    The claim then follows formally from this.
\end{proof}

\subsection{Comparison to the theory of zips} \hfill\\ \label{subsec:comp-zips}
In this subsection we compare our $(\calG, \bmmu)$-displays with certain zips as considered by Shen, Yu and Zhang in \cite[Section 3]{shen-yu-zhang}.

\begin{notation}
    Throughout this subsection $R$ always denotes on $\Fpbar$-algebra.
\end{notation}

To conveniently formulate the material in this section, we use the following ad hoc definition of a zip datum that is essentially the same as in \cite[Definition 3.1]{pink-wedhorn-ziegler-zip-data} and \cite[Definition 3.1]{pink-wedhorn-ziegler-zips}, but slightly more general.

\begin{definition}
    A \emph{zip datum} is a tuple $\calZ = (H, P_+, P_-, L_-, \varphi)$ where $H$ is an affine group scheme over $\Fpbar$, $P_+, P_- \subseteq H$ are subgroup schemes, $L_-$ is a quotient group scheme of $P_-$ and $\varphi \colon P_+ \to L_-$ is a morphism of group schemes.

    Let $\calZ$ be a zip datum as above.
    A \emph{zip for $\calZ$ over $R$} is a tuple $(I, I_+, I_-, \Psi)$ where $I$ is an $H$-torsor over $R$, $I_+$ (resp.\ $I_-$) is a reduction of $I$ to a $P_+$-torsor (resp.\ a $P_-$-torsor) and
    \[
        \Psi \colon L_- \times^{\varphi, P_+} I_+ \to L_- \times^{P_-} I_-
    \]
    is an isomorphism of $L_-$-torsors.
\end{definition}

\begin{proposition}
    Let $\calZ = (H, P_+, P_-, L_-, \varphi)$ be a zip datum and let
    \[
        E_{\calZ} \coloneqq P_+ \times_{\varphi, L_-} P_-
    \]
    be the associated zip group.

    The groupoids of zips for $\calZ$, for varying $\Fpbar$-algebras, form a stack that is naturally equivalent to the quotient stack $[E_{\calZ} \backslash H]$ for the action
    \[
        E_{\calZ} \times H \to H, \qquad \roundbr[\big]{(p_+, p_-), h} \mapsto p_+ h p_-^{-1}.
    \]
\end{proposition}

\begin{proof}
    See \cite[Proposition 3.11]{pink-wedhorn-ziegler-zips}.
\end{proof}

\begin{example} \label{ex:zip-datum}
    Let $H_0$ be a connected reductive group over $\Fp$ and let
    \[
        \mu \colon \bfG_{m, \Fpbar} \to H_{0, \Fpbar}
    \]
    be a minuscule cocharacter.
    Then we have an associated zip datum $(H, P_+, P_-, L_-, \varphi)$ that is given as follows.
    \begin{itemize}
        \item
        $H \coloneqq G_{\Fpbar}$.

        \item
        $P_+ \coloneqq P_{\mu^{-1}} \subseteq H$.

        \item
        $P_- \coloneqq P_{\mu}^{\sigma} \subseteq H$.

        \item
        $L_- = (P_-)^{\rdt}$ is the reductive quotient of $P_-$.
        Note that $\mu$ gives a Levi decomposition $P_{\mu} = U_{\mu} \rtimes L_{\mu}$ so that we can identify $L_-$ with $L_{\mu}^{\sigma}$ and thus also with $(P_+)^{\rdt, \sigma}$.

        \item
        $\varphi \colon P_+ \to L_-$ is the composition
        \[
            \varphi \colon P_+ \to (P_+)^{\rdt} \xrightarrow{\sigma} (P_+)^{\rdt, \sigma} \cong L_-.
        \]
    \end{itemize}
    Zips for this zip datum are then the same as zips for $(H_0, \mu^{-1})$ as in \cite{pink-wedhorn-ziegler-zips} (see also \cite[Subsection 1.2]{zhang}).
\end{example}

Recall that $\rmM^{\loc}_{\Fpbar}$ decomposes into finitely many $\calG$-orbits that are called Kottwitz-Rapoport (KR) strata and are enumerated by the set $\Adm(\mu)_{\calG}$ (see \cite{he-rapoport}).
Let us fix $w \in \Adm(\mu)_{\calG}$ and denote the corresponding KR-stratum by $\rmM^{\loc, w} \subseteq \rmM^{\loc}_{\Fpbar}$.

\begin{definition}
    We say that a $(\calG, \mu)$-pair $(\calP, q)$ over $R$ is \emph{of type $w$} if the morphism $q \colon \calP \to \rmM^{\loc}_{\Fpbar}$ has image inside $\rmM^{\loc, w}$ and write $\Pair_{\calG, \mu, \Fpbar}^w \subseteq \Pair_{\calG, \mu, \Fpbar}$ for the substack of $(\calG, \mu)$-pairs of type $w$.

    We make the same definition for $(\calG, \mu)$-displays and the truncated variants.
\end{definition}

Let us now also fix $M_1^{\std} \subseteq \Lambda_{\Zpbrev}$, making $(\Lambda_{\Zpbrev}, M_1^{\std})$ into a $(\calG, \bmmu)$-pair of type $w$ over $\Fpbar$, and an isomorphism $\xi \colon \widetilde{M_1^{\std}} \to \Lambda_{\Zpbrev}$.
Note that the tuple $(M_1^{\std}, \xi)$ defines a point $w \in \rmM^{\loc, (\infty)}(\Fpbar) \subseteq G(\Qpbrev)$ (see \Cref{subsec:comp-local-shtukas}).
We then have the following data.

\begin{itemize}
    \item
    Write $H_{w, +} \subseteq (L^+ \calG)_{\Fpbar}$ for the stabilizer of $w \in \rmM^{\loc}(\Fpbar)$.
    Note that $H_{w, +}$ naturally is the automorphism group scheme of the $(\calG, \bmmu)$-pair $(\Lambda_{\Zpbrev}, M_1^{\std})$.
    We also write $H_{w, +}^{(m)} \subseteq (L^{(m)} \calG)_{\Fpbar}$ for the image of $H_{w, +}$.

    \item
    Write $H_{w, -} \subseteq (L^+ \calG)_{\Fpbar}$ for the stabilizer of the \emph{conjugate filtration}, i.e.\ the image of the composition
    \[
        \Lambda_{\Zpbrev} \to \widetilde{M_1^{\std}} \xrightarrow{\xi} \Lambda_{\Zpbrev} \to \Lambda_{\Fpbar}.
    \]
    Here the first map is the one from \Cref{rmk:pairs-inclusions} and the action of $(L^+ \calG)_{\Fpbar}$ on $\Lambda_{\Fpbar}$ is the usual one through the quotient $\calG_{\Fpbar}$.
    We also write $H_{w, -}^{(n)} \subseteq (L^{(n)} \calG)_{\Fpbar}$ for the image of $H_{w, -}$ and $H_{w, -}^{(m, n)} \subseteq (L^{(m)} \calG)_{\Fpbar}$ for the preimage of $H_{w, -}^{(n)}$.
    This last bit of notation is a bit awkward as we always have $H_{w, -}^{(m, n)} = H_{w, -}^{(m)}$ when $n \neq \onerdt$.

    \item
    We now consider the morphism
    \[
    \begin{array}{r c l}
        \varphi \colon H_{w, +}
        & \to
        & (L^+ \calG)_{\Fpbar},
        \\
        h
        & \mapsto
        & \roundbr[\Big]{\Lambda_{W(R)} \xrightarrow{\xi^{-1}} \roundbr[\big]{\widetilde{M_1^{\std}}}_{W(R)} \xrightarrow{\widetilde{h}} \roundbr[\big]{\widetilde{M_1^{\std}}}_{W(R)} \xrightarrow{\xi} \Lambda_{W(R)}}.
    \end{array}
    \]
    One can characterize $\varphi$ by the property
    \[
        \varphi(h) = \sigma(w^{-1} h w) \in G \roundbr[\big]{W(R)[1/p]}
    \]
    for all $h \in H_{w, +}(R)$.
    
    Also note that in fact $\varphi$ has image inside $H_{w, -}$.
    Indeed, given $h \in H_{w, +}(R)$ the automorphism $\widetilde{h}$ will always stabilize the image of $\Lambda_{W(R)} \to (\widetilde{M_1^{\std}})_{W(R)}$.
    In the following we will always think of $\varphi$ as a morphism $\varphi \colon H_{w, +} \to H_{w, -}$.

    The morphism $\varphi$ clearly induces morphisms $H_{w, +}^{(m)} \to H_{w, -}^{(n)}$ that we again denote by $\varphi$.
\end{itemize}

At this point we have constructed zip data
\[
    \roundbr[\big]{(L^+ \calG)_{\Fpbar}, H_{w, +}, H_{w, -}, H_{w, -}, \varphi} \quad \text{and} \quad \roundbr[\big]{(L^{(m)} \calG)_{\Fpbar}, H_{w, +}^{(m)}, H_{w, -}^{(m, n)}, H_{w, -}^{(n)}, \varphi}.
\]

\begin{remark} \label{rmk:zip-reductive}
    When $\calG$ is reductive there is only one KR-stratum and we can choose $w = \mu(p)$ for some representative of $\mu$ defined over $\calO_F$ for some finite unramified extension $F/\Qp$.
    Then $H_{w, +}$ and $\varphi$ as defined here are precisely the base changes to $\Fpbar$ of the objects $H_{\mu}$ and $\varphi$ from Notation \ref{not:bueltel-pappas}.
\end{remark}

\begin{proposition}
    There is a natural equivalence
    \[
        \Pair_{\calG, \mu, \Fpbar}^w \cong \curlybr[\big]{\text{$H_{w, +}$-torsors}}, \qquad (M, M_1) \mapsto \iIsom \roundbr[\big]{(M, M_1), (\Lambda_{\Zpbrev}, M_1^{\std})_R}.
    \]
    Under this equivalence the construction $(M, M_1) \mapsto \widetilde{M_1}$ naturally identifies with $\calQ \mapsto (L^+ \calG)_{\Fpbar} \times^{\varphi, H_{w, +}} \calQ$.
    Consequently we obtain a natural equivalence
    \[
        \Disp_{\calG, \mu, \Fpbar}^w \cong \curlybr[\big]{\text{zips for $\roundbr[\big]{(L^+ \calG)_{\Fpbar}, H_{w, +}, H_{w, -}, H_{w, -}, \varphi}$}}.
    \]

    Similarly we also have natural equivalences
    \[
        \Pair_{\calG, \mu, \Fpbar}^{(m), w} \cong \curlybr[\big]{\text{$H_{w, +}^{(m)}$-torsors}}
    \]
    and
    \[
        \Disp_{\calG, \mu, \Fpbar}^{(m, n), w} \cong \curlybr[\big]{\text{zips for $\roundbr[\big]{(L^{(m)} \calG)_{\Fpbar}, H_{w, +}^{(m)}, H_{w, -}^{(m, n)}, H_{w, -}^{(n)}, \varphi}$}}.
    \]
\end{proposition}

\begin{proof}
    This is immediate from the definitions.
\end{proof}

In \cite[Section 1.3]{shen-yu-zhang} the authors construct a zip datum $\calZ_w = (\calG_{\Fpbar}^{\rdt}, P_+, P_-, L_-, \varphi)$ associated to $w$.
As in \cite[Section 3.3]{shen-yu-zhang} we have natural morphisms of zip data
\[
    \roundbr[\big]{(L^+ \calG)_{\Fpbar}, H_{w, +}, H_{w, -}, H_{w, -}, \varphi}, \roundbr[\big]{(L^{(m)} \calG)_{\Fpbar}, H_{w, +}^{(m)}, H_{w, -}^{(m, n)}, H_{w, -}^{(n)}, \varphi} \to \calZ_w
\]
that are compatible in the obvious sense.
If we write $\Zip_{\calG, \mu, w}$ for the stack of zips for $\calZ_w$ we thus obtain natural (compatible) smooth morphisms of algebraic stacks
\[
    \Disp_{\calG, \mu, \Fpbar}^{(m, n), w} \to \Zip_{\calG, \mu, w}.
\]
For $n = \onerdt$ this morphism is a homeomorphism on the level of underlying topological spaces by \cite[Proposition 4.2.6]{shen-yu-zhang}.

\begin{remark}
    When $\calG$ is reductive and $w = \mu(p)$ as in \Cref{rmk:zip-reductive} then the zip datum $\calZ_w$ is the one attached to $(\calG_{\Fp}, \mu)$, see \Cref{ex:zip-datum}.
\end{remark}

\subsection{Application to Shimura varieties}

\begin{notation}
    In this section we change our setup.

    Let $(\bfG, \bfX)$ be a Shimura datum of Hodge type (i.e.\ a Shimura datum that embeds into a Siegel Shimura datum).
    Assume that $G \coloneqq \bfG_{\Qp}$ splits over a tamely ramified extension of $\Qp$ and $p \nmid \abs{\pi_1(G^{\der})}$.
    We then have the associated reflex field $\bfE$ and the associated conjugacy class $\bmmu$ and denote by $E$ the completion of $\bfE$ at the place corresponding to the embedding $\bfE \subseteq \Qbar \to \Qpbar$.
    Let $\bfK_p \subseteq G(\Qp)$ be a parahoric stabilizer and let $\calG$ be the corresponding parahoric group scheme over $\Zp$, so that $\calG_{\Qp} = G$ and $\calG(\Zp) = \bfK_p$.

    By results of Kisin and Pappas (see \cite{kisin-pappas}) there then exist a homogeneous symplectic $\QQ$-vector space $(V, \overline{\psi})$, a self-dual $\Zp$-lattice $\Xi \subseteq V_{\Qp}$ and an embedding of Shimura data
    \[
        (\bfG, \bfX) \to (\GSp(V), S^{\pm})
    \]
    such that $G \to \GSp(V)_{\Qp}$ extends to a closed immersion
    \[
        \calG \to \GSp(\Xi)
    \]
    of group schemes over $\Zp$ and such that $X_{\bmmu}(G) \to X_{\bmmu_g}(\GSp(V)_{\Qp})_E = \LGrass_{V, E}$ extends to a closed immersion of local models
    \[
        \bbM^{\loc}_{\calG, \bmmu} \to (\bbM^{\loc}_{\GSp(\Xi), \bmmu_g})_{\calOE} = \LGrass_{\Xi, \calOE}.
    \]

    Set $\Lambda \coloneqq \Xi^{\vee}$.
    Then we have the local model datum $(\calG, \bmmu)$ and the closed immersion
    \[
        \iota \colon \calG \to \GSp(\Xi) \to \GL(\Xi) \cong \GL(\Lambda)
    \]
    satisfies the assumptions from Notation \ref{not:g-mu-displays-group-setup} (for $(h, d) = (2g, g)$).

    Let $\bfK^p \subseteq \bfG(\adeles^p)$ be a small enough compact open subgroup and set $\bfK \coloneqq \bfK_p \bfK^p \subseteq \bfG(\adeles)$.
    Choose a small enough compact open subgroup $\bfL^p \subseteq \GSp(V)(\adeles^p)$ such that $\bfK^p$ maps into $\bfL^p$ and such that the induced morphism of Shimura varieties
    \[
        \Sh_{\bfK}(\bfG, \bfX) \to \Sh_{\bfL}\roundbr[\big]{\GSp(V), S^{\pm}}_{\bfE}
    \]
    is a closed immersion (where we have set $\bfL_p \coloneqq \GSp(\Xi)(\Zp)$ and $\bfL \coloneqq \bfL_p \bfL^p$).
    Note that this is always possible by \cite[Lemma 2.1.2]{kisin}.
    We then have the Kisin-Pappas integral model $\scrS_{\bfK} = \scrS_{\bfK}(\bfG, \bfX)$ of $\Sh_{\bfK}(\bfG, \bfX)$ over $\calOE$ that is the normalization of the closure of $\Sh_{\bfK}(\bfG, \bfX)_E$ inside $\scrS_{\bfL}(\GSp(V), S^{\pm})_{\calOE}$.
    Here $\scrS_{\bfL}(\GSp(V), S^{\pm})$ is the natural integral model of $\Sh_{\bfL}(\GSp(V), S^{\pm})$ over $\Zp$ coming from its description as a moduli space of principally polarized Abelian varieties.
    Write $\widehat{\scrS_{\bfK}}$ for the $p$-completion of $\scrS_{\bfK}$.

    We assume that Assumption \ref{hyp:grothendieck-messing} is satisfied.
\end{notation}

\begin{construction} \label{con:display-hodge-type}
    Let
    \[
        A \to \scrS_{\bfK}
    \]
    be the pullback of the ($g$-dimensional) Abelian variety up to prime-to-$p$ isogeny over $\scrS_{\bfL}(\GSp(V), S^{\pm})$ coming from its moduli description, and let
    \[
        X \coloneqq A \squarebr[\big]{p^{\infty}}
    \]
    be its associated $p$-divisible group (that is of height $2g$ and dimension $g$).
    By \Cref{thm:p-div-groups-disps} $X$ gives rise to the following data.
    \begin{itemize}
        \item
        A display $(M, M_1, \Psi)$ of type $(2g, g)$ over $\widehat{\scrS_{\bfK}}$.

        \item
        For every point $x \in \scrS_{\bfK}(\Fpbar)$ a Dieudonné display $(M_x, M_{x, 1}, \Psi_x)$ of type $(2g, g)$ over the completed local ring of $\scrS_{\bfK, \calO_{\Ebrev}}$ at $x$.
        These are compatible with $(M, M_1, \Psi)$ in the obvious sense. 
    \end{itemize}
\end{construction}

\begin{theorem}[Hamacher-Kim, Pappas] \label{thm:shimura-variety-g-mu-display}
    There are naturally defined compatible $\calG$-structures on $M$ and $M_x$ for $x \in \scrS_{\bfK}(\Fpbar)$ that make $(M, M_1, \Psi)$ and $(M_x, M_{x, 1}, \Psi_x)$ into (Dieudonné) $(\calG, \bmmu)$-displays.

    Moreover the Dieudonné $(\calG, \bmmu)$-displays $(M_x, M_{x, 1}, \Psi_x)$ are locally universal (see \Cref{rmk:locally-universal}).
\end{theorem}

\begin{proof}
    See \cite[Theorem 8.1.4]{pappas}.
\end{proof}

We can now state and prove our main result.

\begin{theorem} \label{thm:main-result-hodge-type}
    The morphism
    \[
        \text{$\widehat{\scrS_{\bfK}} \to \Disp_{\calG, \bmmu}^{(m, n)}$} \qquad \text{(respectively $\scrS_{\bfK, \Fq} \to \Disp_{\calG, \bmmu}^{(m, \onerdt)}$ for $n = \onerdt$)}
    \]
    induced by $(M, M_1, \Psi)$ is smooth.
\end{theorem}

\begin{proof}
    Both $\scrS_{\bfK}$ and $\Disp_{\calG, \mu}^{(m, n)}$ are of finite type over $\Spf(\calO_E)$.
    Thus by \cite[Tag 02HX]{stacks-project} it suffices to show that every lifting problem
    \[
    \begin{tikzcd}
        \Spec(R) \ar[r] \ar[d]
        & \scrS_{\bfK} \ar[d]
        \\
        \Spec(R') \ar[r] \ar[ru, dashed]
        & \Disp_{\calG, \mu}^{(m, n)},
    \end{tikzcd}
    \]
    where $R' \to R$ is a surjection of Artinian local $\calO_{\Ebrev}$-algebras with residue field $\Fpbar$, admits a solution.
    This follows from \Cref{thm:shimura-variety-g-mu-display}, i.e.\ the local universality of the Dieudonné $(\calG, \mu)$-displays $(M_x, M_{x, 1}, \Psi_x)$, together with \Cref{lem:restriction-smooth-dieudonne}.
\end{proof}

Using \Cref{thm:main-result-hodge-type} we can give new proofs of the following results of Shen, Yu and Zhang.

\begin{corollary}[{\cite[Theorem 4.4.3]{shen-yu-zhang}}] \label{cor:main-result-perfection}
    There is a natural perfectly smooth morphism
    \[
        \scrS_{\bfK, \Fq}^{\pf} \to \Sht_{\calG, \bmmu}^{\loc, (m, n)}
    \]
    (where we allow $n = \onerdt$); see Subsection \ref{subsec:comp-local-shtukas} for the definition of $\Sht_{\calG, \bmmu}^{\loc, (m, n)}$.
\end{corollary}

\begin{proof}
    This follows from \Cref{thm:main-result-hodge-type} and \Cref{prop:comparison-local-shtukas}.
\end{proof}

\begin{corollary}[{\cite[Theorem 3.1.2]{zhang}, \cite[Theorem 3.4.11]{shen-yu-zhang}}] \label{cor:main-result-kr-local}
    Fix an element $w \in \Adm(\bmmu)_{\calG}$ and a representative of $w$ in $\rmM^{\loc, (\infty)}(\Fpbar)$.
    Then there is a natural smooth morphism
    \[
        \scrS_{\bfK, \Fpbar}^w \to \Zip_{\calG, \bmmu, w}
    \]
    where $\scrS_{\bfK, \Fpbar}^w \subseteq \scrS_{\bfK, \Fpbar}$ denotes the KR-stratum corresponding to $w$; see Subsection \ref{subsec:comp-zips} for the definition of $\Zip_{\calG, \bmmu, w}$.
\end{corollary}

\begin{proof}
    This follows from \Cref{thm:main-result-hodge-type} and the discussion in Subsection \ref{subsec:comp-zips}.
\end{proof}

\begin{remark}
    The underlying topological space
    \[
        \abs[\Big]{\roundbr[\big]{\Disp_{\calG, \bmmu}^{(m, \onerdt)}}_{\Fpbar}} = \abs[\Big]{\roundbr[\big]{\Sht_{\calG, \bmmu}^{\loc, (m, \onerdt)}}_{\Fpbar}}
    \]
    has finitely many points whose closure relations can be described explicitly, see \cite[Lemma 4.2.4]{shen-yu-zhang}.
    The (finitely many) fibers of
    \[
        \scrS_{\bfK, \Fpbar} \to \roundbr[\big]{\Disp_{\calG, \bmmu}^{(m, \onerdt)}}_{\Fpbar}
    \]
    are precisely the EKOR strata introduced by He and Rapoport in \cite{he-rapoport}.
    Thus \Cref{thm:main-result-hodge-type} solves the problem of realizing the EKOR-stratification as a smooth morphism from the special fiber of the integral model $\scrS_{\bfK}$ to a naturally defined algebraic stack.
\end{remark}

    \printbibliography

\end{document}